\documentclass[12pt,twosides]{amsart}
\usepackage{amssymb,amsmath,amsthm, amscd, enumerate, mathrsfs}
\usepackage{graphicx, hhline}
\usepackage[all]{xy}
\usepackage{braket}
\usepackage[usenames]{color}
\usepackage{hyperref}
\usepackage{fancyhdr}
\usepackage[top=30truemm,bottom=30truemm,left=30truemm,right=30truemm]{geometry}

\hypersetup{colorlinks=true}

\title{Relations between two log minimal models of log canonical pairs}
\author{Kenta Hashizume}
\date{2020/08/24, version 0.31}
\address{Graduate School of Mathematical Sciences, 
The University of Tokyo, 3-8-1 Komaba Meguro-ku Tokyo 153-8914, Japan}
\email{hkenta@ms.u-tokyo.ac.jp}

\newtheorem{thm}{Theorem}[section]

\newtheorem{lem}[thm]{Lemma}
\newtheorem{cor}[thm]{Corollary}
\newtheorem{prop}[thm]{Proposition}

\theoremstyle{definition}
\newtheorem{defn}[thm]{Definition}
\newtheorem{rem}[thm]{Remark}

\newtheorem*{const}{Construction}
\newtheorem{exam}[thm]{Example}
\newtheorem*{ack}{Acknowledgments} 
\newtheorem*{divisor}{Divisors and morphisms} 
\newtheorem*{sing}{Singularities of pairs} 
\newtheorem*{g-pair}{Generalized pairs} 
\newtheorem*{model}{Models}
\newtheorem*{log-mmp}{Log MMP}

\newtheorem{step1}{Step}

\newtheorem*{claim*}{Claim}
\begin{document}

\maketitle

\begin{abstract}
We study relations between two log minimal models of a fixed lc pair. 
For any two log minimal models of an lc pair constructed with log MMP, we prove that there are small birational models of the log minimal models which can be connected by a sequence of flops, and the two log minimal models share some  properties. 
We also give examples of two log minimal models of an lc pair which have different properties. 
\end{abstract}

\tableofcontents

\section{Introduction}
Throughout this paper we work over the complex number field $\mathbb{C}$, and log minimal model means log minimal model in the classical sense (\cite[Definition 3.50]{kollar-mori}). 

For any projective lc pair whose log canonical divisor is pseudo-effective, it is expected that there exists at least one log minimal model of the lc pair. 
The main technique to construct a log minimal model is the log minimal model program (log MMP, for short), and it is conjectured that all log MMP terminate after finitely many steps. 
In the case of surfaces, it is known that log minimal model of any lc surface is unique.  
On the other hand, when the dimension of the underlying variety of an lc pair is greater than two, log minimal models of the lc pair are not uniquely determined. 
Therefore, it is natural to investigate relations between resulting log minimal models of log MMP starting with a fixed lc pair. 
It is known by Birkar--Cascini--Hacon--M\textsuperscript{c}Kernan that any two $\mathbb{Q}$-factorial log minimal models of a klt pair with a big boundary divisor are connected by a sequence of flops (\cite[Corollary 1.1.3]{bchm}). 
On the other hand, Kawamata \cite{kawamata-flop} proved that any birational map between two $\mathbb{Q}$-factorial terminal pairs over a variety with relatively nef log canonical divisors can be decomposed into a sequence of flops. 
In this paper, we aim to generalize these results to not necessarily $\mathbb{Q}$-factorial lc pairs. 
Recently, the author and Hu \cite{hashizumehu} proved that for any not necessarily $\mathbb{Q}$-factorial lc pair over a variety, if its $\mathbb{Q}$-factorial dlt model has a log minimal model then we can run a log MMP for the lc pair terminating with a log minimal model. 
So it is natural to consider generalizations to lc pairs. 

For any two log minimal models $(X,\Delta)$ and $(X',\Delta')$ of a given lc pair constructed by running log MMP, the naturally induced birational map $\phi\colon X\dashrightarrow X'$ has three important properties: $\phi$ is small (i.e., isomorphic in codimension one), $\phi_{*}\Delta=\Delta'$, and there is an open subset $U\subset X$ such that $\phi$ is an isomorphism on $U$ and all lc centers of $(X,\Delta)$ intersect $U$ (see Lemma \ref{lem--mmp-basic}). Taking this into consideration, in this paper, we deal with lc pairs $(X,\Delta)$ and $(X',\Delta')$ over a variety $Z$ such that $K_{X}+\Delta$ and $K_{X'}+\Delta'$ are nef over $Z$ and there is a birational map $\phi\colon X\dashrightarrow X'$ satisfying the three conditions stated above. 

The following theorems are main results of this paper.

\begin{thm}\label{thm--main}
Let $\pi\colon X\to Z$ and $\pi'\colon X'\to Z$ be projective morphisms from normal quasi-projective varieties to a quasi-projective variety $Z$, and let $(X,\Delta)$ and $(X',\Delta')$ be lc pairs such that $K_{X}+\Delta$ and $K_{X'}+\Delta'$ are nef over $Z$. 
Suppose that there is a small birational map $\phi\colon X\dashrightarrow X'$ over $Z$ such that 
\begin{itemize}
\item
$\Delta'=\phi_{*}\Delta$, and
\item
there is an open subset $U\subset X$ such that $\phi$ is an isomorphism on $U$ and all lc centers of $(X,\Delta)$ intersect $U$. 
\end{itemize}

Then, there are projective small birational morphisms $f\colon \overline{X}\to X$ and $f'\colon\overline{X}'\to X'$ from normal quasi-projective varieties such that $f$ and $f'$ are compositions of extremal contractions and the induced birational map $f'^{-1}\circ \phi \circ f\colon \overline{X} \dashrightarrow \overline{X}'$ is a composition of flops for $K_{\overline{X}}+f^{-1}_{*}\Delta$ over $Z$
\begin{equation*}
\xymatrix@C=21pt@R=16pt
{
\overline{X}\ar@{}[r]|*{=}&\overline{X}_{0}\ar@{-->}^{\varphi_{0}}[rr]\ar[dr]&&\overline{X}_{1}\ar@{-->}[r]\ar[dl]&\cdots \ar@{-->}[r]&\overline{X}_{i}\ar@{-->}^{\varphi_{i}}[rr]\ar[dr]&&\overline{X}_{i+1}\ar@{-->}[r]\ar[dl]&\cdots \ar@{-->}[r]&\overline{X}_{l}\ar@{}[r]|*{=}&\overline{X}'\\
&&V_{0}&&&&V_{i}
}
\end{equation*}
satisfying the following property: 
\begin{itemize}\item[($*$)]
For any $0\leq i<l$, the birational morphisms $\overline{X}_{i}\to V_{i}$ and $\overline{X}_{i+1}\to V_{i}$ of the flop are extremal contractions. 
\end{itemize}
In particular, $\overline{X}_{i}$ is $\mathbb{Q}$-factorial if and only if $\overline{X}_{i+1}$ is $\mathbb{Q}$-factorial for any $0\leq i<l$, and each $\varphi_{i}$ induces an isomorphic linear map $\varphi_{i*}\colon N^{1}(\overline{X}_{i}/Z)_{\mathbb{R}}\to N^{1}(\overline{X}_{i+1}/Z)_{\mathbb{R}}$. 
\end{thm}

\begin{thm}\label{thm--pic0}
Let $\pi\colon X\to Z$, $\pi'\colon X'\to Z$, $(X,\Delta)$, $(X',\Delta')$ and $\phi\colon X\dashrightarrow X'$ be as in Theorem \ref{thm--main}. 
Then the followings hold true.

\begin{itemize}
\item
$X$ has a small $\mathbb{Q}$-factorialization if and only if $X'$  has a small $\mathbb{Q}$-factorialization, 
\item
there is an isomorphism $R^{p}\pi_{*}\mathcal{O}_{X}\overset{\sim}{\longrightarrow} R^{p}\pi'_{*}\mathcal{O}_{X'}$ of sheaves for any $p>0$, and
\item 
for any Cartier divisor $D$ on $X$ such that $D \equiv_{Z} r(K_{X}+\Delta)$ for an $r\in \mathbb{R}$, the birational transform $\phi_{*}D$ is Cartier and $\phi_{*}D \equiv_{Z} r(K_{X'}+\Delta')$. 
\end{itemize}
\end{thm}
For definition of extremal contractions, see Definition \ref{defn--extcont}. 
We emphasize that $D$ and $\phi_{*}D$ in the third assertion of Theorem \ref{thm--pic0} are Cartier. 
To prove these results, we apply a result of non-$\mathbb{Q}$-factorial log MMP (\cite[Theorem 1.7]{hashizumehu}). 
For details, see Section \ref{sec3}.  

By Theorem \ref{thm--pic0}, we obtain the following result. 

\begin{cor}\label{cor--cartierindex}
Let $X\to Z$ be a projective morphism from a normal quasi-projective variety to a quasi-projective variety and $(X,\Delta)$ an lc pair such that $\Delta$ is a boundary  $\mathbb{Q}$-divisor. 
Let $(X,\Delta)\dashrightarrow (Y,\Gamma)$ and $(X,\Delta)\dashrightarrow (Y',\Gamma')$ be two sequences of steps of the $(K_{X}+\Delta)$-MMP over $Z$ to log minimal models. 
Then the followings hold true. 
\begin{itemize}
\item
the Cartier index of $K_{Y}+\Gamma$ coincides with that of $K_{Y'}+\Gamma'$, and
\item
when $Z$ is a point, ${\rm Pic}(Y)_{\mathbb{R}}\simeq N^{1}(Y)_{\mathbb{R}}$ if and only if ${\rm Pic}(Y')_{\mathbb{R}}\simeq N^{1}(Y')_{\mathbb{R}}$. 
\end{itemize}
\end{cor}

Unfortunately, there are two (not necessarily $\mathbb{Q}$-factorial) log minimal models of an lc pair which cannot be connected by a sequence of flops (see Example \ref{exam--1}). 
In addition to this example, in Section \ref{sec4}, we introduce various examples related to the main results. 
In Example \ref{exam--2}, we give a flop for an lc pair which does not satisfy the condition ($*$) in Theorem \ref{thm--main} (see also Remark \ref{rem--inverseflop}). 
In Example \ref{exam--3}, we give two projective lc pairs with nef log canonical divisors such that the two lc pairs are isomorphic in codimension one but the assertion of Theorem \ref{thm--main} and the three assertions of Theorem \ref{thm--pic0} do not hold (for the second assertion of Theorem \ref{thm--pic0}, see also Example \ref{exam--4}). 
In particular, Example \ref{exam--3} shows that to prove the same conclusions as in Theorem \ref{thm--main} and Theorem \ref{thm--pic0}, the hypothesis on the existence of an open subset is necessary. 
Related to Corollary \ref{cor--cartierindex}, in Example \ref{exam--6}, we construct two projective lc pairs with nef log canonical divisors such that the pairs are isomorphic in codimension one and Cartier indices of the log canonical divisors are different. 
Ideas of examples \ref{exam--2}, \ref{exam--3}, \ref{exam--4} and \ref{exam--6} are inspired by \cite[Exercise 96]{kollar-exercise} (see also \cite[Example 7.7]{fujino-remarks}), and these examples are similar to \cite[Example 3.13.9]{fujino-book} by Fujino, which shows that flops do not always preserve dlt property.  

The contents of this paper are as follows: In Section \ref{sec2}, we collect some definitions. 
In Section \ref{sec3}, we prove Theorem \ref{thm--main}, Theorem \ref{thm--pic0} and Corollary \ref{cor--cartierindex}. 
In Section \ref{sec4}, we give some examples.

\begin{ack}
The author was partially supported by JSPS KAKENHI Grant Number JP19J00046. 
He thanks Professor Yoshinori Gongyo, Doctor Masaru Nagaoka and Doctor Shou Yoshikawa for discussions and comments. 
He thanks Professor Osamu Fujino for comments and informing him of \cite[Example 3.13.9]{fujino-book}. 
He thanks the referee for a lot of suggestions. 
\end{ack}

\section{Preliminaries}\label{sec2}

\begin{divisor}
Let $\pi\colon X\to Z$ be a projective morphism from a normal variety to a variety, and let 
$D_{1}$ and $D_{2}$ be two $\mathbb{R}$-Cartier divisors on $X$. 
Then $D_{1}$ and $D_{2}$ are {\em numerically equivalent over $Z$}, denoted by $D_{1}\equiv_{Z}D_{2}$, if $(D_{1}\cdot C)=(D_{2}\cdot C)$ for all curves $C\subset X$ contained in a fiber of $\pi$. 
When $D_{1}\equiv_{Z}0$, we say $D_{1}$ is {\em numerically trivial over $Z$}. 
Let $N^{1}(X/Z)_{\mathbb{R}}$ be the $\mathbb{R}$-vector space consisting of all $\mathbb{R}$-Cartier divisors on $X$ modulo numerical equivalence over $Z$. 
Then the {\em relative Picard number} $\rho(X/Z)$ is defined by the dimension of $N^{1}(X/Z)_{\mathbb{R}}$ as an $\mathbb{R}$-vector space. 
It is known that $\rho(X/Z)<\infty$. 
When $Z$ is a point we remove $Z$ in the above notation, i.e., $D_{1}\equiv D_{2}$, $D_{1}\equiv 0$, $N^{1}(X)_{\mathbb{R}}$, and $\rho(X)$. 

\begin{defn}[Contraction and extremal contraction,  see also {\cite[Definition 3.34]{kollar-mori}}]\label{defn--extcont}
Let $f\colon X\to Y$ be a projective morphism from a normal variety to a variety. 
Then $f$ is a {\em contraction} if $f$ is surjective and has connected fibers. 
A contraction $f$ is an {\em extremal contraction} if for any two Cartier divisors $D_{1}$ and $D_{2}$, there are $a_{1},a_{2}\in \mathbb{Z}$ which are not both zero and a Cartier divisor $D_{Y}$ on $Y$ such that $a_{1}D_{1}-a_{2}D_{2}\sim f^{*}D_{Y}$. 
\end{defn}
\end{divisor}

For any surjective morphism $X\to Z$ of normal projective varieties, we have $\rho(X)-\rho(Z)\geq \rho(X/Z)$. 
In general, the equality does not hold (see, for example, \cite[V, Exercise 1.6]{hartshorne}).

\begin{sing}
A {\em pair} $(X,\Delta)$ consists of a normal variety $X$ and a boundary $\mathbb{R}$-divisor $\Delta$ on $X$ such that $K_{X}+\Delta$ is $\mathbb{R}$-Cartier. 
For any pair $(X,\Delta)$ and any prime divisor $P$ over $X$, we denote by $a(P,X,\Delta)$ the discrepancy of $P$ with respect to $(X,\Delta)$. 
In this paper, we use the standard definitions of Kawamata log terminal (klt, for short) pair, log canonical (lc, for short) pair and divisorially log terminal (dlt, for short) pair as in \cite{bchm}. 
When $(X,\Delta)$ is an lc pair, an {\em lc center} of $(X,\Delta)$ is the image on $X$ of a prime divisor $P$ over $X$ whose discrepancy $a(P,X,\Delta)$ is equal to $-1$. 
\end{sing}

\begin{model} We define some models used in this paper. 

\begin{defn}\label{defn--minimalmodel}
In this paper, weak lc models, log minimal models and log canonical models are defined in the classical sense (\cite[Definition 3.50]{kollar-mori}). 
We write down the definitions of those models. 

Let $(X,\Delta)$ be an lc pair and $X\to Z$ a projective morphism to a variety. 
Let $X'\to Z$ be a projective morphism from a normal variety and let $\phi\colon X\dashrightarrow X'$ be a birational contraction over $Z$ such that $K_{X'}+\phi_{*}\Delta$ is $\mathbb{R}$-Cartier. 
Put $\Delta'=\phi_{*}\Delta$. 
Then the pair $(X',\Delta')$ is a {\em weak log canonical model} ({\em weak lc model}, for short) of $(X,\Delta)$ over $Z$ if 
\begin{itemize}
\item
$K_{X'}+\Delta'$ is nef over $Z$, and 
\item
for any prime divisor $D$ on $X$ which is exceptional over $X'$, we have
\begin{equation*}
a(D, X, \Delta) \leq a(D, X', \Delta').
\end{equation*}
\end{itemize}
A weak lc model $(X',\Delta')$ of $(X,\Delta)$ over $Z$ is a {\it log minimal model} if 
\begin{itemize}
\item
the above inequality on discrepancies is strict. 
\end{itemize}
A log minimal model $(X',\Delta')$ of $(X,\Delta)$ over $Z$ is a {\it good minimal model} if 
\begin{itemize}
\item
$K_{X'}+\Delta'$ is semi-ample over $Z$. 
\end{itemize}
A weak lc model $(X',\Delta')$ of $(X,\Delta)$ over $Z$ is a {\it log canonical model} if 
\begin{itemize}
\item
$K_{X'}+\Delta'$ is ample over $Z$. 
\end{itemize}
In this paper, we do not assume log minimal models to be $\mathbb{Q}$-factorial or dlt. 
\end{defn}

For any lc pair $(X,\Delta)$ on a normal quasi-projective variety $X$, we can construct a {\em dlt blow-up} $(Y,\Gamma)\to (X,\Delta)$ of $(X,\Delta)$ as in \cite[Theorem 10.4]{fujino-fund} or \cite[Theorem 3.1]{kollarkovacs}.

\begin{defn}[$D$-flip]
Let $X\to Z$ be a projective morphism from a normal variety to a variety and $D$ an $\mathbb{R}$-Cartier divisor on $X$. 
A birational morphism $f\colon X\to V$ over $Z$, where $V$ is projective over $Z$ and normal, is a $D$-{\em flipping contraction over $Z$} if 
\begin{itemize}
\item
$\rho(X/V)=1$, 
\item
$f$ is small, i.e., $f$ is isomorphic in codimension one, and
\item
$-D$ is ample over $V$. 
\end{itemize}
Let $f'\colon X'\to V$ be a projective birational morphism over $Z$ from a normal variety $X'$, and let $\varphi\colon X\dashrightarrow X'$ be the induced birational map. 
Then $f'$ is a {\em flip} of $f$ if 
\begin{itemize}
\item
$f'$ is small, and
\item
$\varphi_{*}D$ is $\mathbb{R}$-Cartier and ample over $V$. 
\end{itemize}
We sometimes call the whole diagram $D$-flip.
\end{defn}

\begin{defn}[flop {(cf.~\cite[Definition 7.5]{fujino-remarks})} and symmetric flop]\label{defn--flop}
Let $X\to Z$ be a projective morphism from a normal variety to a variety, and let $(X,\Delta)$ be an lc pair. 
In this paper, a {\em flop for $K_{X}+\Delta$ over $Z$} consists of the following diagram 
\begin{equation*}
\xymatrix@R=16pt
{
X\ar@{-->}^{\varphi}[rr]\ar[dr]^{f}\ar[ddr]&&X' \ar[dl]_{f'}\ar[ddl]\\
&V\ar[d]\\
&Z
}
\end{equation*}
such that
\begin{itemize}
\item
$f$ is a $D$-flipping contraction over $Z$ for some $\mathbb{R}$-Cartier divisor $D$ on $X$, 
\item
$f'$ is a flip of $f$, and
\item
$K_{X}+\Delta \equiv_{V}0$.
\end{itemize}
In this paper, with notation as above, a flop for $K_{X}+\Delta$ over $Z$ is a {\em symmetric flop} if 
\begin{itemize}
\item
both $f$ and $f'$ are extremal contractions. 
\end{itemize}
\end{defn}

In general, flops are not always symmetric flops even if $Z={\rm Spec}\mathbb{C}$, $D = K_{X}+B$ for an lc pair $(X,B)$ and $\rho(X)=\rho(X')=\rho(V)+1$ (see Remark \ref{rem--inverseflop}).

The following lemma gives a sufficient condition for flops being symmetric. 

\begin{lem}\label{lem--suffcond}
With notation as in Definition \ref{defn--flop}, suppose that 
\begin{itemize}
\item
$Z$ is quasi-projective, 
\item
$\rho(X/Z)=\rho(X'/Z)$, and 
\item
there is an effective $\mathbb{R}$-Cartier divisor $E$ on $X$ such that $(X,\Delta+E)$ is lc and $-(K_{X}+\Delta+E)$ is ample over $V$. 
\end{itemize}
Then the flop is a symmetric flop. 
Furthermore, there is an effective $\mathbb{R}$-Cartier divisor $F'$ on $X'$ such that $(X',\varphi_{*}\Delta+F')$ is lc and $-(K_{X'}+\varphi_{*}\Delta+F')$ is ample over $V$. 
\end{lem}

\begin{proof}[{\rm \bf Proof}]
Since $Z$ is quasi-projective, $(X,\Delta+E)$ is lc and $-(K_{X}+\Delta+E)$ is ample over $V$, we can find an $\mathbb{R}$-Cartier divisor $H \geq 0$ on $X$ such that $(X,\Delta+E+H)$ is lc and $K_{X}+\Delta+E+H\equiv_{V}0$. 
Since $\rho(X/V)=1$, by applying the standard argument of convex geometry and \cite[Theorem 13.1]{fujino-fund}, for any $\mathbb{R}$-Cartier divisor $G$ on $X$ we can find a real number $r$ such that $G-rD$ is $\mathbb{R}$-linearly equivalent to the pullback of an $\mathbb{R}$-Cartier divisor on $V$. 
By taking the birational transform on $X'$, we see that $\varphi_{*}G-r\varphi_{*}D$ is $\mathbb{R}$-linearly equivalent to the pullback of an $\mathbb{R}$-Cartier divisor on $V$. 
Since $\varphi_{*}D$ is $\mathbb{R}$-Cartier, $\varphi_{*}G$ is also $\mathbb{R}$-Cartier. 
In this way, for any $\mathbb{R}$-Cartier divisor $G$ on $X$ the birational transform $\varphi_{*}G$ is $\mathbb{R}$-Cartier. 
Then $\varphi$ induces a linear map $\varphi_{*}\colon N^{1}(X/Z) \to N^{1}(X'/Z)$. 
Since $\varphi$ is small, by the negativity lemma, the linear map is injective. 
Then this is an isomorphism because $\rho(X/Z)=\rho(X'/Z)$ by hypothesis. 

We first find an $\mathbb{R}$-divisor $F'\geq 0$ of Lemma \ref{lem--suffcond}. 
By definition of $H$, the divisor $K_{X'}+\varphi_{*}(\Delta+E+H)$ is $\mathbb{R}$-Cartier. 
Furthermore, we have $K_{X'}+\varphi_{*}(\Delta+E+H) \equiv_{V}0$, and the pair $(X',\varphi_{*}(\Delta+E+H))$ is lc by the negativity lemma. 
By definition, $-E$ is ample over $V$, so $E\sim_{\mathbb{R},V} \alpha D$ for some $\alpha >0$. 
Then $\varphi_{*}E\sim_{\mathbb{R},V} \alpha \varphi_{*}D$. 
Since $\varphi_{*}D$ is ample over $V$, we see that $\varphi_{*}E$ is ample over $V$. 
Therefore, the divisor
$$-(K_{X'}+\varphi_{*}\Delta+\varphi_{*}H)=-(K_{X'}+\varphi_{*}(\Delta+E+H))+\varphi_{*}E$$
is ample over $V$. 
Since $\varphi_{*}E$ is effective, we see that the pair $(X',\varphi_{*}(\Delta+H))$ is lc. 
Putting $F'=\varphi_{*}H$, we get the desired effective $\mathbb{R}$-divisor. 

Next, we prove that the flop is a symmetric flop. 
We have seen that for any $\mathbb{R}$-Cartier divisor $G$ on $X$ we can find a real number $r$ such that $G-rD\sim_{\mathbb{R},V}0$. 
From this, it is easy to see that $f\colon X\to V$ is an extremal contraction. 
Let $D'_{1}$ and $D'_{2}$ be any Cartier divisors on $X'$. 
Since $\varphi_{*}\colon N^{1}(X/Z) \to N^{1}(X'/Z)$ is an isomorphism, we may find $\mathbb{Q}$-Cartier divisors $B_{1}$ and $B_{2}$ on $X$ such that $\varphi_{*}B_{1}\equiv_{Z}D'_{1}$ and $\varphi_{*}B_{2}\equiv_{Z}D'_{2}$. 
Since $(X',\varphi_{*}\Delta+F')$ is lc and $-(K_{X'}+\varphi_{*}\Delta+F')$ is ample over $V$ for some $F'\geq 0$, by \cite[Theorem 13.1]{fujino-fund}, we obtain $\varphi_{*}B_{1}\sim_{\mathbb{Q},V}D'_{1}$ and $\varphi_{*}B_{2}\sim_{\mathbb{Q},V}D'_{2}$. 
Furthermore, since $f$ is an extremal contraction, there are rational numbers $a_{1},a_{2}$ which are not both zero such that $a_{1}B_{1}\sim_{\mathbb{Q},V} a_{2}B_{2}$. 
Then
$$a_{1}D'_{1}\sim_{\mathbb{Q},V}a_{1}\varphi_{*}B_{1}\sim_{\mathbb{Q},V} a_{2}\varphi_{*}B_{2}\sim_{\mathbb{Q},V}D'_{2}.$$
This shows that $f'\colon X'\to V$ is an extremal contraction. 
From these discussion, $f$ and $f'$ are both extremal contractions, and we see that the flop is a symmetric flop. 
\end{proof}

The following lemma states properties of symmetric flops. 

\begin{lem}\label{lem--symflop}
With notation as in Definition \ref{defn--flop}, the following properties hold true for symmetric flops. 
\begin{itemize}
\item
$\rho(X/V)=\rho(X'/V)=1$, 
\item
for any $\mathbb{R}$-divisor $G$ on $X$, the birational transform $\varphi_{*}G$ on $X'$ is $\mathbb{R}$-Cartier if and only if $G$ is $\mathbb{R}$-Cartier, in particular, $X$ is $\mathbb{Q}$-factorial if and only if $X'$ is $\mathbb{Q}$-factorial, and
\item
the map $\varphi$ induces an isomorphic linear map $\varphi_{*}\colon N^{1}(X/Z)_{\mathbb{R}}\to N^{1}(X'/Z)_{\mathbb{R}}$, in particular, the equality $\rho(X/Z)=\rho(X'/Z)$ holds. 
\end{itemize}
\end{lem}

\begin{proof}[{\rm \bf Proof}]
The first property is clear. 
To prove the second and the third properties, we apply proof of Lemma \ref{lem--suffcond}.
Pick an $\mathbb{R}$-divisor $G$ on $X$, and suppose that $G$ is $\mathbb{R}$-Cartier. 
As in the proof of Lemma \ref{lem--suffcond}, we can find a real number $r$ such that $G-rD$ is $\mathbb{R}$-linearly equivalent to the pullback of an $\mathbb{R}$-Cartier divisor on $V$. 
Then $\varphi_{*}G-r\varphi_{*}D$ is $\mathbb{R}$-linearly equivalent to the pullback of an $\mathbb{R}$-Cartier divisor on $V$, therefore $\varphi_{*}G$ is $\mathbb{R}$-Cartier. 
The converse can be proved in the same way. 
Thus, the second property holds. 
For the third property, we can easily check from the second property that there are natural linear maps $\varphi_{*}\colon N^{1}(X/Z)_{\mathbb{R}}\to N^{1}(X'/Z)_{\mathbb{R}}$ and $\varphi^{-1}_{*}\colon N^{1}(X'/Z)_{\mathbb{R}}\to N^{1}(X/Z)_{\mathbb{R}}$, and they are inverse linear maps of each other. 
\end{proof}

\end{model}

\begin{log-mmp}
As in \cite[4.9]{fujino-book}, we may run non-$\mathbb{Q}$-factorial log MMP for not necessarily $\mathbb{Q}$-factorial lc pairs. 
We recall some basic properties of non-$\mathbb{Q}$-factorial log MMP. 

\begin{rem}[see also {\cite[Remark 6.1]{hashizumehu}}]\label{rem--mmp}
Let $(X_{1},\Delta_{1})$ be an lc pair, and let $X_{1}\to Z$ be a projective morphism from a normal quasi-projective variety to a quasi-projective variety. 
Let
\begin{equation*}
(X_{1},\Delta_{1})\dashrightarrow (X_{2},\Delta_{2})\dashrightarrow \cdots \dashrightarrow (X_{i},\Delta_{i})\dashrightarrow \cdots
\end{equation*}
be a sequence of steps of a $(K_{X_{1}}+\Delta_{1})$-MMP over $Z$. 
\begin{enumerate}
\item[(1)]
The map $X_{i}\dashrightarrow X_{i+1}$ is a birational contraction for any $i\geq 1$. 
\item[(2)]
By the cone theorem, for any $i$ and any $\mathbb{Q}$-Cartier divisor $D_{i}$ on $X_{i}$, the birational transform $D_{i+1}$ on $X_{i+1}$ is $\mathbb{Q}$-Cartier. 
\item[(3)]
By (2) and the negativity lemma, if the birational map $X_{i}\dashrightarrow X_{i+1}$ is small then it induces a linear map $N^{1}(X_{i}/Z)_{\mathbb{R}}\to N^{1}(X_{i+1}/Z)_{\mathbb{R}}$ by taking the birational transform. 
\item[(4)]
If $X_{i}$ is $\mathbb{Q}$-factorial for some $i \geq 1$, then $X_{j}$ are $\mathbb{Q}$-factorial for all $j> i$. 
\end{enumerate}
\end{rem}

The following lemma follows from construction of log MMP. 

\begin{lem}\label{lem--mmp-basic}
Let $(X,\Delta)$ be an lc pair, and let $X\to Z$ be a projective morphism from a normal quasi-projective variety to a quasi-projective variety. 
Let $(X,\Delta)\dashrightarrow (Y,\Gamma)$ and $(X,\Delta)\dashrightarrow (Y',\Gamma')$ be two sequences of steps of the $(K_{X}+\Delta)$-MMP over $Z$ to log minimal models. 
Then the induced birational map $\phi\colon Y\dashrightarrow Y'$ satisfies the followings.  
\begin{itemize}
\item
$\phi$ is small, 
\item
$\Gamma'=\phi_{*}\Gamma$, and 
\item
there is an open subset $U\subset Y$ such that $\phi$ is an isomorphism on $U$ and all lc centers of $(Y,\Gamma)$ intersect $U$. 
\end{itemize}
\end{lem}

\begin{proof}[{\rm \bf Proof}]
We denote the log MMP $X\dashrightarrow Y$ (resp.~$X\dashrightarrow Y'$) by $\varphi$ (resp.~$\varphi'$). For any prime divisor $P$ over $X$, denote the center of $P$ on $X$ (resp.~$Y$, $Y'$) by $c_{X}(P)$ (resp.~$c_{Y}(P)$, $c_{Y'}(P)$). Let $V\subset X$ (resp.~$V'\subset X$) be the largest open subset on which $\varphi$ (resp.~$\varphi'$) is an isomorphism. Since $K_{Y}+\Gamma$ and $K_{Y'}+\Gamma'$ are nef, by the negativity lemma, we have $a(P,Y,\Gamma)=a(P,Y',\Gamma')$ for any prime divisor $P$ over $Y$. 
By definition of log minimal models, a prime divisor $D$ on $X$ is contracted by $\varphi$ if and only if $a(D,X,\Delta)<a(D,Y,\Gamma)$. 
The same property holds for $a(D,Y',\Gamma')$ and $\varphi'$. 
In particular, $D$ is contracted by $\varphi$ if and only if $D$ is contracted by $\varphi'$. So the induced birational map $\phi =\varphi'\circ\varphi^{-1}\colon Y\dashrightarrow Y'$ is small, and $\Gamma'=\phi_{*}\Gamma$. Pick any prime divisor $P$ over $Y$ with $a(P,Y,\Gamma)=-1$. Then $a(P,X,\Delta)=a(P,Y',\Gamma')=-1$. Construction of log MMP shows that discrepancy of any prime divisor over $X$ strictly increases by the $(K_{X}+\Delta)$-MMP if and only if the center of the prime divisor on $X$ is contained in non-isomorphic locus of the $(K_{X}+\Delta)$-MMP, therefore, the equality $a(P,X,\Delta)=a(P,Y,\Gamma)$ shows $c_{X}(P)\cap V\neq \emptyset$. The same property holds for $a(P,Y',\Gamma')$ and $c_{X}(P)\cap V'$. So we have $c_{X}(P)\cap (V\cap V')\neq \emptyset$. Put $U=\varphi(V\cap V')$. Then $\phi$ is an isomorphism on $U$. Since $c_{X}(P)\cap (V\cap V')\neq \emptyset$ for any prime divisor $P$ over $Y$ with $a(P,Y,\Gamma)=-1$, we see that $U$ intersects all lc centers of $(Y,\Gamma)$. Hence $(Y,\Gamma)$ and $(Y',\Gamma')$ over $Z$ satisfy the three properties of Lemma \ref{lem--mmp-basic}. 
\end{proof}

We close this section with a result of non-$\mathbb{Q}$-factorial log MMP for lc pairs, which plays an important role in this paper. 

\begin{thm}[cf.~{\cite[Theorem 1.7]{hashizumehu}}]\label{thm--mmp}
Let $\pi\colon X\to Z$ be a projective morphism from a normal quasi-projective variety to a quasi-projective variety, and let $(X,\Delta)$ be an lc pair. 
Suppose that $(X,\Delta)$ has a weak lc model over $Z$ in the sense of Definition \ref{defn--minimalmodel}. 
  
Then there is a sequence of birational contractions 
\begin{equation*}
(X,\Delta)=(X_{1},\Delta_{1})\dashrightarrow (X_{2},\Delta_{2})\dashrightarrow \cdots \dashrightarrow (X_{l},\Delta_{l})
\end{equation*}
of a non-$\mathbb{Q}$-factorial $(K_{X}+\Delta)$-MMP over $Z$ that terminates with a log minimal model $(X_{l},\Delta_{l})$ as in Definition \ref{defn--minimalmodel}. 
\end{thm}

\begin{proof}[{\rm \bf Proof}]
We note that weak lc models as in Definition \ref{defn--minimalmodel} are in particular weak lc models in the sense of \cite[Definition 2.6]{has-mmp} and \cite[Section 2]{hashizumehu}. 
If $(X,\Delta)$ has a weak lc model over $Z$ as in Definition \ref{defn--minimalmodel}, by \cite[Remark 2.10]{has-mmp}, we see that $(X,\Delta)$ has a log minimal model over $Z$ in the sense of \cite[Definition 2.6]{has-mmp} and \cite[Section 2]{hashizumehu}. 
Then Theorem \ref{thm--mmp} follows from \cite[Theorem 1.7]{hashizumehu}. 
\end{proof}

\end{log-mmp}

\section{Proofs of main results}\label{sec3}

\begin{prop}\label{prop--smallmodif}
Let $X\to Z$ and $X'\to Z$ be two projective morphisms from normal quasi-projective varieties to a quasi-projective variety $Z$, and let $(X,\Delta)$ and $(X',\Delta')$ be lc pairs such that $K_{X}+\Delta$ and $K_{X'}+\Delta'$ are nef over $Z$. 
Suppose that there is a small birational map $\phi\colon X\dashrightarrow X'$ over $Z$ such that 
\begin{itemize}
\item
$\Delta'=\phi_{*}\Delta$, and
\item
there is an open subset $U\subset X$ such that $\phi$ is an isomorphism on $U$ and all lc centers of $(X,\Delta)$ intersect $U$. 
\end{itemize}

Then, there exists a projective small birational morphism $\psi\colon\overline{X}\to X$ from a normal quasi-projective variety such that 
\begin{itemize}
\item
$\psi$ is an isomorphism over $U$, and  
\item 
there is an ample $\mathbb{R}$-divisor $A'\geq 0 $ on $X'$ such that $(X',\Delta'+A')$ is lc and if we set $\overline{A}$ as the birational transform of $A'$ on $\overline{X}$, then $K_{\overline{X}}+\psi^{-1}_{*}\Delta+\overline{A}$ is $\mathbb{R}$-Cartier and the pair $(\overline{X},\psi^{-1}_{*}\Delta+\overline{A})$ is lc.  
\end{itemize}
\end{prop}

\begin{proof}[{\rm \bf Proof}]
Let $f\colon Y \to X$ be a log resolution of $(X,\Delta)$ such that the induced birational map $g\colon Y\dashrightarrow X'$ is a morphism. 
Let $\Gamma$ be the sum of $f^{-1}_{*}\Delta$ and the reduced $f$-exceptional divisor. 
Let $A'\geq0$ be an ample $\mathbb{R}$-divisor on $X'$ such that the pairs $(X',\Delta'+A')$ and $(Y,\Gamma+g^{*}A')$ are lc. 
By running a $(K_{Y}+\Gamma)$-MMP over $X$ with scaling of an ample divisor, we construct a dlt blow-up $f'\colon (Y',\Gamma')\to (X,\Delta)$. 
Let $A_{Y'}$ be the birational transform of $g^{*}A'$ on $Y'$. 
Then the map $Y\dashrightarrow Y'$ is a sequence of steps of a $(K_{Y}+\Gamma+t_{0}g^{*}A')$-MMP for some $t_{0}\in (0,1)$. 
So there is $t\in (0,t_{0})$ such that $(Y',\Gamma'+tA_{Y'})$ is lc and all lc centers of the pair are lc centers of $(Y',\Gamma')$. 
Then all lc centers of $(Y',\Gamma'+tA_{Y'})$ intersect $f'^{-1}(U)$. 

We show that $(Y',\Gamma'+tA_{Y'})$ has the log canonical model over $X$. 
It is sufficient to prove that $(Y',\Gamma'+tA_{Y'})$ has a good minimal model over $X$. 
We set $U_{Y'}=f'^{-1}(U)$. 
By \cite[Theorem 1.2]{has-mmp}, we only need to prove that the pair
$\bigl(U_{Y'}, (\Gamma'+tA_{Y'})|_{U_{Y'}}\bigr)$ has a good minimal model over $U$. 
Since $\phi\colon X\dashrightarrow X'$ is isomorphic on $U$, the divisor $\phi^{-1}_{*}A'|_{U}$ is $\mathbb{R}$-Cartier and $A_{Y'}|_{U_{Y'}}$ is equal to the pullback of $\phi^{-1}_{*}A'|_{U}$ to $U_{Y'}$. 
Then we have
\begin{equation*}\tag{$*$}
K_{U_{Y'}}+(\Gamma'+tA_{Y'})|_{U_{Y'}}=f'|_{U_{Y'}}^{*}(K_{U}+\Delta|_{U}+t\phi^{-1}_{*}A'|_{U}),
\end{equation*} 
thus the pair $\bigl(U_{Y'}, (\Gamma'+tA_{Y'})|_{U_{Y'}}\bigr)$ itself is a good minimal model of $\bigl(U_{Y'}, (\Gamma'+tA_{Y'})|_{U_{Y'}}\bigr)$ over $U$.
Therefore, $(Y',\Gamma'+tA_{Y'})$ has a good minimal model over $X$. 
In this way, we see that $(Y',\Gamma'+tA_{Y'})$ has the log canonical model over $X$. 

Let $(Y',\Gamma'+tA_{Y'})\dashrightarrow (\overline{X}, \Delta_{\overline{X}}+tA_{\overline{X}})$ be the birational contraction over $X$ to the log canonical model, where $\Delta_{\overline{X}}$ (resp.~$A_{\overline{X}}$) is the birational transform of $\Gamma'$ (resp.~$A_{Y'}$) on $\overline{X}$. 
Let $\psi\colon \overline{X}\to X$ be the induced birational  morphism. 
We prove that $\psi$ is isomorphic over $U$ and small. 
By ($*$), the restriction of $K_{\overline{X}}+\Delta_{\overline{X}}+tA_{\overline{X}}$ to $\psi^{-1}(U)$ is numerically trivial over $U$, so $\psi$ is an isomorphism over $U$. 
This also implies that all $f'$-exceptional prime divisors $E$ on $Y'$ are contracted by $Y'\dashrightarrow \overline{X}$ because 
$-1\leq a(E,Y',\Gamma'+tA_{Y'})\leq a(E,Y',\Gamma')=-1$ 
 and all $E$ intersect $f'^{-1}(U)$ by our choice of $t$. 
In this way, we see that $\psi$ is small. 
From these facts, we see that $\psi$ is isomorphic over $U$ and small. 

Finally, we prove that $\psi$ is the desired birational morphism. 
But it is obvious from construction. 
Indeed, $\psi$ is small and it satisfies the first condition of Proposition \ref{prop--smallmodif}. 
Moreover the divisor $tA'$ satisfies the second condition of Proposition \ref{prop--smallmodif}. 
\end{proof}

\begin{lem}\label{lem--picard}
Let $X\to Z$ be a projective morphism from a normal variety to a variety. 
Let $W\to Z$ be a projective morphism from a $\mathbb{Q}$-factorial variety to $Z$ such that there is a birational contraction $W\dashrightarrow X$ over $Z$.  
Then $\rho(W/Z)\geq \rho(X/Z)$. 
\end{lem}

\begin{proof}[{\rm \bf Proof}]
Take a common resolution $f \colon \overline{W}\to W$ and $g \colon \overline{W}\to X$ of the map $W\dashrightarrow X$. 
For any $\mathbb{R}$-Cartier divisors $D_{1}$ and $D_{2}$ on $X$, if $D_{1}- D_{2}\equiv_{Z}0$ then $g^{*}D_{1}-g^{*}D_{2}\equiv_{Z}0$, and $f_{*}g^{*}D_{1}- f_{*}g^{*}D_{2}\equiv_{Z}0$ by the negativity lemma. 
Thus, we can define a linear map $f_{*}g^{*}\colon N^{1}(X/Z)_{\mathbb{R}}\to N^{1}(W/Z)_{\mathbb{R}}$ by $D\mapsto f_{*}g^{*}D$. 
For any $\mathbb{R}$-Cartier divisors $D_{1}$ and $D_{2}$ on $X$, if $f_{*}g^{*}D_{1}\equiv_{Z} f_{*}g^{*}D_{2}$ then we have $D_{1}\equiv_{Z} D_{2}$ since $D_{1}$ and $D_{2}$ are the birational transforms of $f_{*}g^{*}D_{1}$ and $f_{*}g^{*}D_{2}$ on $X$, respectively. 
In this way, we see that the linear map $f_{*}g^{*}\colon N^{1}(X/Z)_{\mathbb{R}}\to N^{1}(W/Z)_{\mathbb{R}}$ is injective, so $\rho(W/Z)\geq \rho(X/Z)$. 
\end{proof}

\begin{lem}\label{lem--extcont}
Let $(X,\Delta)$ be an lc pair and $\psi\colon X' \to X$ a projective small birational morphism of normal quasi-projective varieties.
Suppose that there exists an open subset $U\subset X$ over which $\psi$ is an isomorphism and all lc centers of $(X,\Delta)$ intersect $U$. 

Then the followings hold true.
\begin{itemize}
\item
$R^{p}\psi_{*}\mathcal{O}_{X'}=0$ for any $p>0$,  
\item
$\psi$ is an isomorphism or a composition of extremal contractions, and  
\item
for any Cartier divisor $D'$ on $X'$, if $D'\equiv_{X}0$ then there is a Cartier divisor $D$ on $X$ such that $D'\sim \psi^{*}D$. 
In particular, $\psi_{*}D'$ is Cartier. 
\end{itemize}
\end{lem}

\begin{proof}[{\rm \bf Proof}]
Since $\psi$ is small, we may assume that codimension of $X'\setminus \psi^{-1}(U)$ in $X'$ is at least two. 
First, we find an $\mathbb{R}$-Cartier divisor $G'\geq0$ on $X'$ such that $-G'$ is ample over $X$ and $(X',\psi^{-1}_{*}\Delta+G')$ is lc. 
We note that $K_{X'}+\psi^{-1}_{*}\Delta=\psi^{*}(K_{X}+\Delta)$, so $(X', \psi^{-1}_{*}\Delta)$ is lc and all lc centers of $(X', \psi^{-1}_{*}\Delta)$ intersect $\psi^{-1}(U)$. 
Pick an ample $\mathbb{R}$-divisor $A'$ on $X'$, and pick an ample $\mathbb{R}$-divisor $H$ on $X$. 
Since $\psi$ is an isomorphism over $U$, we can find $s>0$ such that the divisor $(s\psi^{*}H-A')|_{\psi^{-1}(U)}$ is ample. 
Since all lc centers of $(X', \psi^{-1}_{*}\Delta)$ intersect $\psi^{-1}(U)$, by taking the closure of a general member of the $\mathbb{R}$-linear system of $(s\psi^{*}H-A')|_{\psi^{-1}(U)}$, we get an effective $\mathbb{R}$-divisor $ H'\sim_{\mathbb{R}}s\psi^{*}H-A'$ such that ${\rm Supp}H'$ contains no lc centers of $(X', \psi^{-1}_{*}\Delta)$. 
Then $-H'$ is ample over $X$ and there is $t>0$ such that $(X',\psi^{-1}_{*}\Delta+tH')$ is lc. 
Therefore $G':=tH'$ is the desired divisor. 

From now on, we prove three assertions of Lemma \ref{lem--extcont}. 
By construction, the divisor $-(K_{X'}+\psi^{-1}_{*}\Delta+G')$ is ample over $X$. 
Since $(X', \psi^{-1}_{*}\Delta+G')$ is lc, we have $R^{p}\psi_{*}\mathcal{O}_{X'}=0$ for any $p>0$ by \cite[Theorem 5.6.4]{fujino-book}. 
We prove the second and the third assertion of Lemma \ref{lem--extcont} by induction on $\rho(X'/X)$. 
If $\rho(X'/X)=0$, then $\psi$ is an isomorphism and the third assertion is clear. 
So suppose $\rho(X'/X)>0$. 
The pair $(X',\psi^{-1}_{*}\Delta+G')$ is lc and the divisor $K_{X'}+\psi^{-1}_{*}\Delta+G'$ is in particular not nef over $X$. 
By the cone and contraction theorem (\cite[Theorem 4.5.2]{fujino-book}), there is an extremal contraction $f \colon X'\to X''$ over $X$. 
Furthermore, \cite[Theorem 4.5.2 (4) (iii)]{fujino-book} shows that for any Cartier divisor $D'$ on $X'$, if $D'\equiv_{X}0$ then there is a Cartier divisor $D''$ on $X''$ such that $D'\sim f^{*}D''$ and $D''\equiv_{X}0$.  
Since the morphism $\psi\colon X' \to X$ is an isomorphism over $U$ and  small, so is the morphism $X''\to X$. 
We denote $X''\to X$ by $g$. 
It is clear that $\rho(X''/X)<\rho(X'/X)$, therefore we can apply the induction hypothesis to $g$.
In this way, we see that $\psi$ is a composition of extremal contractions, and there is a Cartier divisor $D$ on $X$ such that $D''\sim g^{*}D$. 
Then $D'\sim f^{*}D''\sim f^{*}g^{*}D=\psi^{*}D$. 
Thus the second and the third statement hold.
\end{proof}

\begin{thm}\label{thm--flop}
Let $\pi\colon X\to Z$ and $\pi'\colon X'\to Z$ be projective morphisms from normal quasi-projective varieties to a quasi-projective variety $Z$, and let $(X,\Delta)$ and $(X',\Delta')$ be lc pairs such that $K_{X}+\Delta$ and $K_{X'}+\Delta'$ are nef over $Z$. 
Suppose that there is a small birational map $\phi\colon X\dashrightarrow X'$ over $Z$ such that 
\begin{itemize}
\item
$\Delta'=\phi_{*}\Delta$, and
\item
there is an open subset $U\subset X$ such that $\phi$ is an isomorphism on $U$ and all lc centers of $(X,\Delta)$ intersect $U$. 
\end{itemize}

Then, there are projective small birational morphisms $f\colon \overline{X}\to X$ and $f'\colon\overline{X}'\to X'$ from normal quasi-projective varieties such that $f$ and $f'$ are compositions of extremal contractions and the induced birational map $f'^{-1}\circ \phi \circ f\colon \overline{X} \dashrightarrow \overline{X}'$ is a composition of symmetric flops for $K_{\overline{X}}+f^{-1}_{*}\Delta$ over $Z$. 
\begin{equation*}
\xymatrix@C=21pt
{
\overline{X}\ar@{}[r]|*{=}&\overline{X}_{0}\ar@{-->}^{\varphi_{0}}[rr]\ar[dr]_{g_{0}}&&\overline{X}_{1}\ar@{-->}[r]\ar[dl]^{g'_{0}}&\cdots \ar@{-->}[r]&\overline{X}_{i}\ar@{-->}^{\varphi_{i}}[rr]\ar[dr]_{g_{i}}&&\overline{X}_{i+1}\ar@{-->}[r]\ar[dl]^{g'_{i}}&\cdots \ar@{-->}[r]&\overline{X}_{l}\ar@{}[r]|*{=}&\overline{X}'\\
&&V_{0}&&&&V_{i}
}
\end{equation*}
Furthermore, the following two conditions hold true.
\begin{itemize}
\item[(a)]
$f$ and $f'$ are isomorphisms over $U$ and $\phi(U)$ respectively, and
\item[(b)]
for each $i$ and morphisms $g_{i}\colon\overline{X}_{i}\to V_{i}$ and $g'_{i}\colon \overline{X}_{i+1}\to V_{i}$, there are $\mathbb{R}$-divisors $\overline{B}_{i}$ and $\overline{B}_{i+1}$ on $\overline{X}_{i}$ and $\overline{X}_{i+1}$ respectively such that $(\overline{X}_{i},\overline{B}_{i})$ and $(\overline{X}_{i+1},\overline{B}_{i+1})$ are lc and $-(K_{\overline{X}_{i}}+\overline{B}_{i})$ and $-(K_{\overline{X}_{i+1}}+\overline{B}_{i+1})$ are ample over $V_{i}$. 
\end{itemize}
\end{thm}

Note that the $\mathbb{R}$-divisor $\overline{B}_{i}$ of Theorem \ref{thm--flop} (b) for $g_{i}\colon \overline{X}_{i}\to V_{i}$ is not necessarily the same as the $\mathbb{R}$-divisor of Theorem \ref{thm--flop} (b) for $g'_{i-1}\colon \overline{X}_{i}\to V_{i-1}$.

\begin{proof}[{\rm \bf Proof of Theorem \ref{thm--flop}}]
By taking a common resolution of $\phi\colon X \dashrightarrow X'$ and the negativity lemma, we have $a(P,X,\Delta)=a(P,X',\Delta')$ for any prime divisor $P$ over $X$. 
Therefore, the set $\phi(U)$ intersects all lc centers of $(X',\Delta')$. 
We prove that there are projective small birational morphisms $f\colon \overline{X}\to X$ and $f'\colon\overline{X}'\to X'$ from normal quasi-projective varieties such that $f$ and $f'$ satisfy property (a) of Theorem \ref{thm--flop} and the two lc pairs $(\overline{X}, f^{-1}_{*}\Delta)$ and $(\overline{X}', f'^{-1}_{*}\Delta')$ are connected by a sequence of symmetric flops which satisfy property (b) of Theorem \ref{thm--flop}. 
If we can obtain these birational maps, then $f$ and $f'$ are compositions of extremal contractions by Lemma \ref{lem--extcont}, from which Theorem \ref{thm--flop} holds true. 
In this way, it is sufficient to prove the statement. 
We fix a smooth variety $Y$ projective over $Z$ such that there are birational contractions $Y\dashrightarrow X$ and $Y\dashrightarrow X'$ over $Z$. 
We prove Theorem \ref{thm--flop} by induction on $\rho(Y/Z)-{\rm max}\{\rho(X/Z),\rho(X'/Z)\}$. 

We first consider the case where $\rho(Y/Z)-{\rm max}\{\rho(X/Z),\rho(X'/Z)\}<0$. 
But the case cannot happen because of Lemma \ref{lem--picard}, thus we finish this case. 
In particular, we settle the starting point of the induction. 

From now on, we assume $\rho(Y/Z)-{\rm max}\{\rho(X/Z),\rho(X'/Z)\}\geq0$. 
By switching $(X,\Delta)$ and $(X',\Delta')$ if necessary, we may assume $\rho(X/Z)\geq\rho(X'/Z)$. 
By Proposition \ref{prop--smallmodif}, there is a projective small birational morphism $\psi\colon\bar{X}\to X$ from a normal quasi-projective variety such that by putting $\bar{\Delta}=\psi^{-1}_{*}\Delta$ we have
\begin{itemize}
\item
$\psi$ is an isomorphism over $U$, and  
\item 
there is an ample $\mathbb{R}$-divisor $A'\geq 0 $ on $X'$ such that $(X',\Delta'+A')$ is lc and if we set $\bar{A}$ as the birational transform of $A'$ on $\bar{X}$, then $K_{\bar{X}}+\bar{\Delta}+\bar{A}$ is $\mathbb{R}$-Cartier and the pair $(\bar{X},\bar{\Delta}+\bar{A})$ is lc.  
\end{itemize} 
Note that $(\bar{X},\bar{\Delta})$ is lc and  $K_{\bar{X}}+\bar{\Delta}$ is nef since $K_{\bar{X}}+\bar{\Delta}=\psi^{*}(K_{X}+\Delta)$. 
By construction, the set $\psi^{-1}(U)$ intersects all lc centers of $(\bar{X},\bar{\Delta})$ and the composition $\phi\circ\psi\colon \bar{X}\dashrightarrow X'$ is an isomorphism on $\psi^{-1}(U)$. 
Therefore, the lc pairs $(\bar{X},\bar{\Delta})$ and $(X',\Delta')$ and the map $\phi\circ\psi\colon \bar{X}\dashrightarrow X'$ over $Z$ satisfy the hypothesis of Theorem \ref{thm--flop}. 
It is obvious that 
$$\rho(Y/Z)-{\rm max}\{\rho(\bar{X}/Z),\rho(X'/Z)\} \leq \rho(Y/Z)-{\rm max}\{\rho(X/Z),\rho(X'/Z)\}.$$
From this and argument in the first paragraph, we may replace $(X,\Delta)$ with $(\bar{X},\bar{\Delta})$. 
In this way, replacing $(X,\Delta)$, we may assume that there is an ample $\mathbb{R}$-divisor $A'\geq 0 $ on $X'$ such that $(X',\Delta'+A')$ is lc, $A:=\phi^{-1}_{*}A'$ is $\mathbb{R}$-Cartier and the pair $(X,\Delta+A)$ is lc. 

For any $t\in(0,1)$, the pair $(X',\Delta'+tA')$ is lc and $K_{X'}+\Delta'+tA'$ is ample over $Z$. 
So the pair $(X',\Delta'+tA')$ is the log canonical model of $(X,\Delta+tA)$ over $Z$. 
In particular, $(X,\Delta+tA)$ has a good minimal model over $Z$ for any $t\in (0,1)$. 
By argument of length of extremal rays (\cite[Section 18]{fujino-fund} and \cite[Theorem 4.7.2]{fujino-book}, see also \cite[Remark 2.13]{has-mmp}), we can find $\epsilon \in (0,1)$ such that the birational transform of $K_{X}+\Delta$ is trivial over all extremal contractions of any sequence of steps of the $(K_{X}+\Delta+\epsilon A)$-MMP over $Z$. 
By Theorem \ref{thm--mmp}, there is a sequence of steps of a $(K_{X}+\Delta+\epsilon A)$-MMP over $Z$
$$(X,\Delta+\epsilon A)=(X_{0},\Delta_{0}+\epsilon A_{0})\dashrightarrow (X_{1},\Delta_{1}+\epsilon A_{1})\dashrightarrow \cdots \dashrightarrow (X_{l},\Delta_{l}+\epsilon A_{l})$$
terminating with a good minimal model $(X_{l},\Delta_{l}+\epsilon A_{l})$. 
Then all maps $X_{i}\dashrightarrow X_{i+1}$ are flips since $K_{X}+\Delta+\epsilon A$ is movable over $Z$. 
Since $(X',\Delta'+\epsilon A')$ is the log canonical model of $(X,\Delta+tA)$ over $Z$, there is a natural small birational morphism $\psi'\colon X_{l} \to X'$ such that $\Delta'=\psi'_{*}\Delta_{l}$. 
For each $0\leq i<l$, we denote the birational map $X_{i}\dashrightarrow X_{i+1}$ by $\varphi_{i}$.  
As in Remark \ref{rem--mmp} (3), each birational map $\varphi_{i}$ induces an injective linear map $\varphi_{i*}\colon N^{1}(X_{i}/Z)_{\mathbb{R}}\to N^{1}(X_{i+1}/Z)_{\mathbb{R}}$. 
Then
$$\rho(X/Z)\leq \cdots \leq \rho(X_{i}/Z)\leq \rho(X_{i+1}/Z) \leq \cdots \leq \rho(X_{l}/Z).$$
Now there are two possibilities: $\rho(X/Z)<\rho(X_{l}/Z)$ and $\rho(X/Z)=\rho(X_{l}/Z)$. 

First we assume $\rho(X/Z)<\rho(X_{l}/Z)$. 
Because the stable base locus of $K_{X}+\Delta+\epsilon A$ over $Z$ is contained in $X\setminus U$, by basic properties of log MMP, the induced birational map $X \dashrightarrow X_{l}$ is an isomorphism on $U$, and the image of $U$ on $X_{l}$ is $\psi'^{-1}(\phi(U))$. 
Since $K_{X_{l}}+\Delta_{l}=\psi'^{*}(K_{X'}+\Delta')$, all lc centers of $(X_{l},\Delta_{l})$ intersect $\psi'^{-1}(\phi(U))$. 
In this way, the birational map $(X,\Delta)\dashrightarrow (X_{l},\Delta_{l})$ over $Z$ satisfies the hypothesis of Theorem \ref{thm--flop}. 
Now we have $\rho(X/Z)={\rm max}\{\rho(X/Z),\rho(X'/Z)\}$, therefore
$$\rho(Y/Z)-{\rm max}\{\rho(X/Z),\rho(X_{l}/Z)\}<\rho(Y/Z)-{\rm max}\{\rho(X/Z),\rho(X'/Z)\}.$$ 
The argument in the first paragraph shows that we may replace $(X',\Delta')$ with $(X_{l},\Delta_{l})$. 
Replacing $(X',\Delta')$ with $(X_{l},\Delta_{l})$ and by the induction hypothesis of Theorem \ref{thm--flop}, we see that Theorem \ref{thm--flop} holds in this case. 

Next, we assume $\rho(X/Z)=\rho(X_{l}/Z)$. 
Then $\rho(X_{i}/Z)=\rho(X_{i+1}/Z)$ for any $0\leq i<l$. 
All we have to prove is that $X\dashrightarrow X_{l}$ is a sequence of symmetric flops which satisfy property (b) of Theorem \ref{thm--flop}. 
Note that the  birational map $X\dashrightarrow X_{l}$ is a sequence of flops for $K_{X}+\Delta$ over $Z$ since each $\varphi_{i}\colon X_{i}\dashrightarrow X_{i+1}$ is a flip over $Z$ and $K_{X_{i}}+\Delta_{i}$ is trivial over the extremal contraction of the $(i+1)$-th step of the log MMP. 
Fix $0\leq i<l$ and consider the $(i+1)$-th step of the log MMP
\begin{equation*}
\xymatrix
{
X_{i}\ar@{-->}^{\varphi_{i}}[rr]\ar[dr]_{g_{i}}&&X_{i+1} \ar[dl]^{g'_{i}}\\
&V_{i}
}
\end{equation*}
where $-(K_{X_{i}}+\Delta_{i}+\epsilon A_{i})$ is ample over $V_{i}$ by construction. 
Since $\rho(X_{i}/Z)=\rho(X_{i+1}/Z)$, we can apply Lemma \ref{lem--suffcond}. 
By Lemma \ref{lem--suffcond}, the flops is symmetric and there is $E_{i+1}\geq 0$ on $X_{i+1}$ such that $(X_{i+1},\Delta_{i+1}+E_{i+1})$ is lc and $-(K_{X_{i+1}}+\Delta_{i+1}+E_{i+1})$ is ample over $V_{i}$. 
Then the divisors $\Delta_{i}+\epsilon A_{i}$ and $\Delta_{i+1}+E_{i+1}$ satisfy Theorem \ref{thm--flop} (b). 
So we are done. 
\end{proof}

\begin{proof}[\rm{\bf Proof of Theorem \ref{thm--main}}]
It immediately follows from Theorem \ref{thm--flop} and Lemma \ref{lem--symflop}. 
\end{proof}

\begin{proof}[\rm{\bf Proof of Theorem \ref{thm--pic0}}]
Firstly, we prove the first assertion of Theorem \ref{thm--pic0}. 
We only prove the existence of a small $\mathbb{Q}$-factorialization of $X'$ under the assumption on existence of a small $\mathbb{Q}$-factorialization $f\colon W\to X$ because the another direction can be proved similarly. 
We put $\Gamma=f^{-1}_{*}\Delta$. 
Then $(W,\Gamma)$ is lc since $K_{W}+\Gamma=f^{*}(K_{X}+\Delta)$. 
By taking a common resolution of $\phi$ and the negativity lemma, we have $a(P,X,\Delta)=a(P,X',\Delta')$ for any prime divisor $P$ over $X$. 
Thus the set $\phi(U)$ intersects all lc centers of $(X',\Delta')$.  
Let $A'\geq0$ be an ample $\mathbb{R}$-divisor on $X'$ such that $(X',\Delta'+A')$ is lc, and let $A_{W}$ be the birational transform of $A'$ on $W$. 

We show that there is $\epsilon>0$ such that the pair $(W,\Gamma+\epsilon A_{W})$ is lc.  
It is sufficient to prove that ${\rm Supp}A_{W}$ contains no lc centers of $(W,\Gamma)$. 
Pick any prime divisor $P$ over $W$ with $a(P,W,\Gamma)=-1$. 
Then it is easy to check $a(P,X',\Delta')=-1$. 
We denote the center of $P$ on $W$ (resp.~$X'$) by $c_{W}(P)$ (resp.~$c_{X'}(P)$). 
Then $c_{X'}(P)$ is an lc center of $(X',\Delta')$ and $c_{X'}(P)$ intersects $\phi(U)$. 
Since $(X',\Delta'+A')$ is lc, we have $c_{X'}(P)|_{\phi(U)}\not\subset {\rm Supp}A'|_{\phi(U)}$. 
By hypothesis, the induced birational map $\phi\circ f\colon W\dashrightarrow X'$ is small and it is a morphism over $\phi(U)$. 
Since $f\colon W\to X$ is small, the restriction $A_{W}|_{f^{-1}(U)}$ is equal to the pullback of $A'|_{\phi(U)}$ to $f^{-1}(U)$. 
Then $c_{W}(P)|_{f^{-1}(U)}\not\subset {\rm Supp}A_{W}|_{f^{-1}(U)}$, so any lc center of $(W,\Gamma)$ is not contained in ${\rm Supp}A_{W}$. 
Therefore, we can find $\epsilon>0$ such that $(W,\Gamma+\epsilon A_{W})$ is lc. 

Since $W\dashrightarrow X'$ is small and $K_{X'}+\Delta'+\epsilon A'$ is ample over $Z$, the pair $(X',\Delta'+\epsilon A')$ is the log canonical model of $(W,\Gamma+\epsilon A_{W})$ over $Z$. 
By Theorem \ref{thm--mmp}, we may construct a sequence of steps of a $(K_{W}+\Gamma+\epsilon A_{W})$-MMP over $Z$
$$(W,\Gamma+\epsilon A_{W})\dashrightarrow (W',\Gamma'+\epsilon A_{W'})$$
to a good minimal model over $Z$, and there is a natural birational morphism $W'\to X'$. 
By Remark \ref{rem--mmp} (4) the variety $W'$ is $\mathbb{Q}$-factorial, and $W'\to X'$ is small by construction. 
In this way, we see that $X'$ has a small $\mathbb{Q}$-factorialization. 

From now on, we freely use notation as in Theorem \ref{thm--flop}. 
Next, we prove the second assertion of Theorem \ref{thm--pic0}. 
By property (a) of Theorem \ref{thm--flop} and Lemma \ref{lem--extcont}, we obtain 
$R^{q}f_{*}\mathcal{O}_{\overline{X}}=0$ and $R^{q}f'_{*}\mathcal{O}_{\overline{X}'}=0$ for any $q>0$. 
By relative Leray spectral sequences, there are isomorphisms of sheaves on $Z$
\begin{equation*}\tag{$1$}
R^{p}\pi_{*}\mathcal{O}_{X}\overset{\sim}{\longrightarrow}R^{p}(\pi\circ f)_{*}\mathcal{O}_{\overline{X}}\quad  {\rm and} \quad R^{p}\pi'_{*}\mathcal{O}_{X'}\overset{\sim}{\longrightarrow}R^{p}(\pi' \circ f')_{*}\mathcal{O}_{\overline{X}'}.
\end{equation*}
For each $0\leq i<l$, we denote the induced morphism $V_{i}\to Z$ by $\pi_{i}$. 
To prove the second assertion of Theorem \ref{thm--pic0}, it is sufficient to prove the existence of an isomorphism
\begin{equation*}\tag{$2$}
R^{p}(\pi_{i}\circ g_{i})_{*}\mathcal{O}_{\overline{X}_{i}}\overset{\sim}{\longrightarrow} R^{p}(\pi_{i}\circ g'_{i})_{*}\mathcal{O}_{\overline{X}_{i+1}}
\end{equation*}
for any $p>0$ and each $i$. 
By property (b) of Theorem \ref{thm--flop} and \cite[Theorem 5.6.4]{fujino-book}, we have $R^{q}g_{i*}\mathcal{O}_{\overline{X}_{i}}=R^{q}g'_{i*}\mathcal{O}_{\overline{X}_{i+1}}=0$. 
By relative Leray spectral sequences, we obtain
$$R^{p}(\pi_{i}\circ g_{i})_{*}\mathcal{O}_{\overline{X}_{i}}\overset{\sim}{\longleftarrow}R^{p}\pi_{i*}\mathcal{O}_{V_{i}}\overset{\sim}{\longrightarrow}R^{p}(\pi_{i}\circ g'_{i})_{*}\mathcal{O}_{\overline{X}_{i+1}}
$$
for any $p>0$. 
In this way, we get the isomorphism (2). 
By (1) and (2), the second assertion of Theorem \ref{thm--pic0} holds. 

Finally, we prove the third assertion of Theorem \ref{thm--pic0}. 
Pick any Cartier divisor $D$ on $X$ such that $D\equiv_{Z} r(K_{X}+\Delta)$ for an $r\in \mathbb{R}$. 
For any $0\leq i\leq l$, we denote the birational transforms of $D$ and $\Delta$ on $\overline{X}_{i}$ by $\overline{D}_{i}$ and $\overline{\Delta}_{i}$, respectively. 
Suppose that $\overline{D}_{i}$ is Cartier and $\overline{D}_{i}\equiv_{Z}r(K_{\overline{X}_{i}}+\overline{\Delta}_{i})$. 
Then $\overline{D}_{i}\equiv_{V_{i}}0$ by construction. 
By property (b) of Theorem \ref{thm--flop} and \cite[Theorem 13.1]{fujino-fund}, the divisor $m\overline{D}_{i}$ is $g_{i}$-globally generated for any $m\gg0$, in other words, the morphism $g_{i}^{*} g_{i*}\mathcal{O}_{\overline{X}_{i}}(m\overline{D}_{i})\to \mathcal{O}_{\overline{X}_{i}}(m\overline{D}_{i})$ is surjective for any $m\gg0$. 
Since $\overline{D}_{i}\equiv_{V_{i}}0$, we see that $m\overline{D}_{i}$ is linearly equivalent to the pullback of a Cartier divisor on $V_{i}$ for any $m\gg0$. 
In this way, we can find a Cartier divisor $G_{i}$ on $V_{i}$ such that $\overline{D}_{i}\sim g_{i}^{*}G_{i}$. 
Then we have $\overline{D}_{i+1}\sim g'^{*}_{i}G_{i}$, thus $\overline{D}_{i+1}$ is Cartier, and $\overline{D}_{i+1}\equiv_{Z}r(K_{\overline{X}_{i+1}}+\overline{\Delta}_{i+1})$. 
Since $\overline{D}_{0}=f^{*}D$ is Cartier and $\overline{D}_{0}\equiv_{Z}r(K_{\overline{X}_{0}}+\overline{\Delta}_{0})$, we can check by induction on $i$ that the birational transform $\overline{D}'$ on $\overline{X}'=\overline{X}_{l}$ is Cartier and $\overline{D}'\equiv_{Z}r(K_{\overline{X}'}+\overline{\Delta}')$. 
By property (a) of Theorem \ref{thm--flop}, we may apply Lemma \ref{lem--extcont} to $f'\colon \overline{X}'\to X'$ and $\overline{D}'$, then $\phi_{*}D=f'_{*}\overline{D}'$ is  Cartier. 
Furthermore, we have $\phi_{*}D\equiv_{Z}r(K_{X'}+\Delta')$. 
In this way, the third assertion of Theorem \ref{thm--pic0} holds. 
We complete the proof.
\end{proof}

\begin{proof}[\rm{\bf Proof of Corollary \ref{cor--cartierindex}}]
It follows from Lemma \ref{lem--mmp-basic} and Theorem \ref{thm--pic0}. 
\end{proof}

\section{Examples}\label{sec4}

\begin{exam}\label{exam--1}
We give a simple example of two projective (not necessarily $\mathbb{Q}$-factorial) klt log minimal models $(X,\Delta)$ and $(X',\Delta')$ such that  the pairs cannot be connected by a sequence of flops. 

Let $(X,\Delta)$ be a klt pair such that $K_{X}+\Delta$ is nef, $X$ is not $\mathbb{Q}$-factorial and $\rho(X)=1$. 
For example, take $X$ as a normal projective cone over $\mathbb{P}^{1}\times \mathbb{P}^{1}$ and $(X,\Delta)$ as a klt pair such that $K_{X}+\Delta\sim_{\mathbb{Q}}0$. 
Let $X'$ be a small $\mathbb{Q}$-factorialization of $X$, and let $\Delta'$ be the birational transform of $\Delta$ on $X'$. 
Then $(X,\Delta)$ and $(X',\Delta')$ cannot be connected by flops because $\rho(X)=1$ which implies that there is no non-trivial contraction from $X$.  
\end{exam}

From now on, we construct a diagram to give examples (cf.~\cite[Exercise 96]{kollar-exercise} and \cite[Example 3.13.9]{fujino-book}).

\begin{const}\label{const}
Let $V$ be a smooth Fano variety such that $\rho(V)=1$, and let $W$ be a smooth projective variety such that $K_{W}\sim_{\mathbb{Q}}0$. 
We define $p_{V}\colon V \times W\to V$ and $p_{W}\colon V \times W \to W$ by natural projections. 
Fix an ample $\mathbb{Q}$-divisor $H_{V}\sim_{\mathbb{Q}}-K_{V}$ and fix an ample $\mathbb{Q}$-divisor $H_{W}$ on $W$, then  put $H_{V \times W}=p_{V}^{*}H_{V}+p_{W}^{*}H_{W}$. 
We pick an integer $m>0$ such that  $mH_{V \times W}$ is a very ample Cartier divisor. 
We consider a $\mathbb{P}^{1}$-bundle
$$f\colon Y=\mathbb{P}_{V \times W}(\mathcal{O}_{V \times W}\oplus \mathcal{O}_{V \times W}(-mH_{V \times W}))\to V \times W.$$ 
Let $T$ be the unique section corresponding to $\mathcal{O}_{Y}(1)$, and put $A_{Y}=T+mf^{*}H_{V \times W}$. 
Since $-K_{V\times W}\sim_{\mathbb{Q}} p_{V}^{*}H_{V}$, we obtain $K_{Y}+T+A_{Y}+f^{*}p_{V}^{*}H_{V}\sim_{\mathbb{Q}} 0$. 
By construction, we also see that $A_{Y}$, $A_{Y}+f^{*}p_{V}^{*}H_{V}$ and $A_{Y}+f^{*}p_{W}^{*}H_{W}$ are all semi-ample, $Y$ is smooth, and the pair $(Y, T)$ is lc. 
Let $g\colon Y\to X$, $g_{V}\colon Y\to X_{V}$ and $g_{W}\colon Y\to X_{W}$ be contractions induced by $A_{Y}$, $A_{Y}+f^{*}p_{V}^{*}H_{V}$ and $A_{Y}+f^{*}p_{W}^{*}H_{W}$, respectively. 
These morphisms are isomorphisms outside $T$. 
By calculations of intersection numbers of curves on $Y$ with those three semi-ample divisors, we see that all curves on $Y$ contracted by $g_{V}$ or $g_{W}$ are contracted by $g$. 
So there are birational morphisms $\pi_{V}\colon X_{V}\to X$ and $\pi_{W}\colon X_{W}\to X$ such that $g=\pi_{V}\circ g_{V}=\pi_{W}\circ g_{W}$. 
We have constructed the following diagram.
 \begin{equation*}
\xymatrix@=18pt
{
&V\times W\ar[dl]_{p_{V}}\ar[dr]^{p_{W}}&&&Y\ar[lll]_{f}\ar[dl]_{g_{V}}\ar[dd]^{g}\ar[dr]^{g_{W}}&\\
V&&W&X_{V}\ar[dr]_{\pi_{V}}&&X_{W}\ar[dl]^{\pi_{W}}\\
&&&&X
}
\end{equation*}
By construction, the image of $T$ by $g$ is a point. 
Since the restriction of $A_{Y}+f^{*}p_{V}^{*}H_{V}$ and $A_{Y}+f^{*}p_{W}^{*}H_{W}$ to $T$ are not big, we see that ${\rm dim}T>{\rm dim}g_{V}(T)$ and ${\rm dim}T>{\rm dim}g_{W}(T)$. 
Since $g\colon Y\to X$ is an isomorphism outside $T$, we see that the morphisms $\pi_{V}$ and $\pi_{W}$ are small birational morphisms. 
Then the induced birational map $X_{V}\dashrightarrow X_{W}$ is small. 

The relation $A_{Y}=g^{*}g_{*}A_{Y}$ shows 
$$g_{V} ^{*}g_{V*}A_{Y}=g_{V} ^{*}\pi_{V} ^{*}g_{*}A_{Y}=g^{*}g_{*}A_{Y}=A_{Y}.$$
Similarly, we have $A_{Y}=g_{W}^{*}g_{W*}A_{Y}$. 
Put $H_{X_{V}}=g_{V*}f^{*}p_{V}^{*}H_{V}$ and $H_{X_{W}}=g_{W*}f^{*}p_{W}^{*}H_{W}$. 
It is easy to see $A_{Y}+f^{*}p_{V}^{*}H_{V}=g_{V}^{*}g_{V*}(A_{Y}+f^{*}p_{V}^{*}H_{V})$ since $g_{V}$ is the contraction induced by $A_{Y}+f^{*}p_{V}^{*}H_{V}$. 
Using these two relations
$$A_{Y}=g_{V}^{*}g_{V*}A_{Y}\quad {\rm and} \quad A_{Y}+f^{*}p_{V}^{*}H_{V}=g_{V}^{*}g_{V*}(A_{Y}+f^{*}p_{V}^{*}H_{V}),$$
we see that $H_{X_{V}}$ is $\mathbb{Q}$-Cartier and $f^{*}p_{V}^{*}H_{V}=g_{V}^{*}H_{X_{V}}$. 
Similarly, we see that $H_{X_{W}}$ is $\mathbb{Q}$-Cartier and $f^{*}p_{W}^{*}H_{W}=g_{W}^{*}H_{X_{W}}$. 
Since $H_{V}$ and $H_{W}$ are ample, both $H_{X_{V}}$ and $H_{X_{W}}$ are semi-ample, and contractions induced by $H_{X_{V}}$ and $H_{X_{W}}$ are morphisms $h_{V}\colon X_{V}\to V$ and $h_{W}\colon X_{W}\to W$ satisfying $p_{V}\circ f=h_{V}\circ g_{V}$ and $p_{W}\circ f=h_{W}\circ g_{W}$, respectively. 
Then $H_{X_{V}}=h_{V}^{*}H_{V}$ and $H_{X_{W}}=h_{W}^{*}H_{W}$. 

We have constructed the following diagram
\begin{equation*}
\xymatrix@R=22pt
{
&&Y\ar[dll]_{p_{V}\circ f}\ar[dl]^{g_{V}}\ar[dr]_{g_{W}}\ar[drr]^{p_{W}\circ f}\\
V&X_{V}\ar[l]^{h_{V}}\ar[dr]_{\pi_{V}}\ar@{-->}[rr]&&X_{W}\ar[dl]^{\pi_{W}}\ar[r]_{h_{W}}&W\\
&&X
}
\end{equation*}
and $\mathbb{Q}$-Cartier divisors 
\begin{equation*}
\begin{split}
&\bullet H_{V}\sim_{\mathbb{Q}}-K_{V}, \quad \bullet H_{X_{V}}=g_{V*}(p_{V}\circ f)^{*}H_{V},\\\ &\bullet H_{W},\quad\quad\quad \quad\quad \bullet H_{X_{W}}=g_{W*}(p_{W}\circ f)^{*}H_{W},\quad {\rm and}\\ 
&\bullet A_{Y}=T+m(p_{V}\circ f)^{*}H_{V}+m(p_{W}\circ f)^{*}H_{W}
\end{split}
\end{equation*}
satisfying
\begin{enumerate}
\item
$Y$ is smooth and $(Y, T)$ is lc, 
\item
$K_{Y}+T+A_{Y}+(p_{V}\circ f)^{*}H_{V}\sim_{\mathbb{Q}} 0$, 
\item
$A_{Y}=g_{V}^{*}g_{V*}A_{Y}=g_{W}^{*}g_{W*}A_{Y}=(\pi_{W}\circ g_{W})^{*}(\pi_{W}\circ g_{W})_{*}A_{Y}$, and 
 \item $(p_{V}\circ f)^{*}H_{V}=g_{V}^{*}H_{X_{V}}$ and $(p_{W}\circ f)^{*}H_{W}=g_{W}^{*}H_{X_{W}}$. 
\end{enumerate}
By construction, the following property also holds.
\begin{equation*}\tag{$\spadesuit$}
g_{V*}A_{Y}+H_{X_{V}}, \;\;g_{W*}A_{Y}+H_{X_{W}} {\rm \;\;and}\;\; (\pi_{V}\circ g_{V})_{*}A_{Y} \;\;{\rm are\;\;ample.} 
\end{equation*}

We investigate properties of diagram consisting of $\pi_{V}\colon X_{V}\to X$, $\pi_{W}\colon X_{W}\to X$ and the map $X_{V}\dashrightarrow X_{W}$ over $X$.

\begin{prop}\label{prop--property}
With notation as in Construction, the followings hold true.
\begin{enumerate}
\item[(i)] \label{prop4.2--a}
The pair $(X_{W},0)$ is $\mathbb{Q}$-factorial klt and $\rho(X_{W})=\rho(W)+1$, 
\item[(ii)] \label{prop4.2--b}
the divisors $-g_{W*}(p_{V}\circ f)^{*}H_{V}$ and $H_{X_{V}}$ are both ample over $X$, 
\item[(iii)] \label{prop4.2--c}
the equality ${\rm dim}H^{p}(X_{W},\mathcal{O}_{X_{W}})={\rm dim}H^{p}(W,\mathcal{O}_{W})$ holds for any $p>0$, and 
\item[(iv)] \label{prop4.2--d}
the variety $X_{V}$ is lc Fano. 
\end{enumerate}
\end{prop}

\begin{proof}[{\rm \bf Proof}]
We divide the proof into several steps. 
\begin{step1}\label{step1--const}
Firstly, we show (i). 

Pick any Weil divisor $D$ on $X_{W}$. 
Then $g_{W*}^{-1}D$ is Cartier, and it is linearly equivalent to the sum of a multiple of $T$ and the pullback of a Cartier divisor on $V\times W$ (\cite[Excercise III, 12.5]{hartshorne}). 
Therefore, we can find a Cartier divisor $D'$ on $V\times W$ such that $D\sim g_{W*}f^{*}D'$ as Weil divisors. 
Recall that $V$ is a smooth Fano variety with $\rho(V)=1$, so $H^{1}(V,\mathcal{O}_{V})=0$ and therefore any $\mathbb{Q}$-Cartier divisor on $V$ is $\mathbb{Q}$-linearly equivalent to some multiple of $H_{V}$. 
Since $W$ is smooth, we can find $\alpha,\beta\in \mathbb{Z}$ and a Cartier divisor $D_{W}$ on $W$ such that $\alpha \neq 0$ and  $\alpha D'\sim \beta p_{V}^{*}H_{V}+p_{W}^{*}D_{W}$ as Cartier divisors (\cite[III, Exercise 12.6 (b)]{hartshorne}). 
Then
$$ \alpha D\sim g_{W*}f^{*}(\beta p_{V}^{*}H_{V}+p_{W}^{*}D_{W})= \beta g_{W*}(p_{V}\circ f)^{*}H_{V}+h_{W}^{*}D_{W}$$ 
as Weil divisors. 
Therefore, to show that $X_{W}$ is $\mathbb{Q}$-factorial, we only need to show that $g_{W*}(p_{V}\circ f)^{*}H_{V}$ is $\mathbb{Q}$-Cartier. 
Now the divisors $g_{W*}A_{Y}=m\bigl(g_{W*}(p_{V}\circ f)^{*}H_{V}+H_{X_{W}}\bigr)$ and  $H_{X_{W}}$ are both $\mathbb{Q}$-Cartier. 
Thus, we see that $g_{W*}(p_{V}\circ f)^{*}H_{V}$ is $\mathbb{Q}$-Cartier. 
Then $D$ is $\mathbb{Q}$-Cartier, and therefore $X_{W}$ is $\mathbb{Q}$-factorial. 

By the above argument, any $\mathbb{R}$-Cartier divisor on $X_{W}$ is $\mathbb{R}$-linearly equivalent to an $\mathbb{R}$-linear combination of $g_{W*}(p_{V}\circ f)^{*}H_{V}$ and the pullback of an $\mathbb{R}$-Cartier divisor on $W$. 
From this, we have $\rho(X_{W})\leq \rho(W)+1$. 
On the other hand, the existence of the morphism $h_{W}\colon X_{W}\to W$ shows that $\rho(X_{W})>\rho(W)$. 
Therefore $\rho(X_{W})= \rho(W)+1$. 

We put $B_{X_{W}}=g_{W*}(p_{V}\circ f)^{*}H_{V}$. 
Then 
$$K_{Y}+T+A_{Y}+(p_{V}\circ f)^{*}H_{V}=g_{W}^{*}(K_{X_{W}}+g_{W*}A_{Y}+B_{X_{W}})$$
 by (2). 
 By (3), we obtain $K_{Y}+T+(p_{V}\circ f)^{*}H_{V}=g_{W}^{*}(K_{X_{W}}+B_{X_{W}})$. 
Furthermore, we have $g_{W}^{*}B_{X_{W}}=\tfrac{1}{m}T+(p_{V}\circ f)^{*}H_{V}$ because
\begin{equation*}
\begin{split}
g_{W }^{*}B_{X_{W }}&=g_{W }^{*}(\tfrac{1}{m}g_{W *}A_{Y}-H_{X_{W }})=\tfrac{1}{m}A_{Y}-g_{W }^{*}H_{X_{W }}=\tfrac{1}{m}T+(p_{V}\circ f)^{*}H_{V}
\end{split}
\end{equation*}
by definition of $A_{Y}$ and (4). 
Hence we see that $K_{Y}+(1-\tfrac{1}{m})T=g_{W}^{*}K_{X_{W}}$. 
By (1), the pair $(Y,(1-\tfrac{1}{m})T)$ is klt, so $(X_{W},0)$ is klt. 
\end{step1}

\begin{step1}\label{step2--const} 
Secondary, we show (ii). 
Note that the divisor $H_{X_{V}}$ is the birational transform of $g_{W*}(p_{V}\circ f)^{*}H_{V}$ on $X_{V}$. 

We use the notation $B_{X_{W}}:=g_{W*}(p_{V}\circ f)^{*}H_{V}$ in Step \ref{step1--const}. 
By (3) and definitions of  $H_{X_{W }}$ and $A_{Y}$, we obtain
$m(B_{X_{W }}+H_{X_{W }})=g_{W*}A_{Y}\sim_{\mathbb{Q},X}0.$ 
Moreover, the divisor $H_{X_{W }}$ is ample over $X$ since $H_{X_{W }}+g_{W *}A_{Y}$ is ample by ($\spadesuit$) and $H_{X_{W }}+g_{W *}A_{Y}\sim_{\mathbb{Q},X}H_{X_{W }}$ by (3). 
Hence $-B_{X_{W }}$ is ample over $X$. 
We also have $H_{X_{V}}\sim_{\mathbb{Q},X}H_{X_{V }}+g_{V *}A_{Y}$ by (3). 
By ($\spadesuit$), we see that $H_{X_{V}}$ is ample over $X$. 
\end{step1}

\begin{step1}\label{step3--const}
Thirdly, we show (iii). 

By applying Leray spectral sequence to $h_{W}\colon X_{W}\to W$, it is sufficient to prove that $R^{q}h_{W*}\mathcal{O}_{X_{W}}=0$ for any $q>0$. 
Since $(X_{W},0)$ is klt, by Kawamata--Viehweg vanishing theorem, it is sufficient to prove that $-K_{X_{W}}$ is nef and big over $W$. 
The argument in Step \ref{step1--const} shows $K_{Y}+(1-\tfrac{1}{m})T=g_{W}^{*}K_{X_{W}}$. 
By (2), we have 
$$K_{Y}+T+A_{Y}+(p_{V}\circ f)^{*}H_{V}+(p_{W}\circ f)^{*}H_{W}\sim_{\mathbb{Q},W} 0.$$
Then, by construction of $A_{Y}$, we obtain 
$$-(K_{Y}+(1-\tfrac{1}{m})T)\sim_{\mathbb{Q},W}A_{Y}+\tfrac{1}{m}T+(p_{V}\circ f)^{*}H_{V}+(p_{W}\circ f)^{*}H_{W}=(1+\tfrac{1}{m})A_{Y}.$$ 
Then $-(K_{Y}+(1-\tfrac{1}{m})T)$ is nef and big over $W$ by (3) and ($\spadesuit$). 
In this way, we see that $-K_{X_{W}}$ is also nef and big over $W$ and therefore $R^{q}h_{W*}\mathcal{O}_{X_{W}}=0$ for any $q>0$. 
Then we have ${\rm dim}H^{p}(X_{W},\mathcal{O}_{X_{W}})={\rm dim}H^{p}(W,\mathcal{O}_{W})$ for any $p>0$. 
\end{step1}

\begin{step1}\label{step4--const}
Finally, we prove (iv).   

By (2) and definition of $H_{X_{V}}$, we have $K_{X_{V}}+g_{V*}A_{Y}+H_{X_{V}}\sim_{\mathbb{Q}}0$. 
Since $g_{V*}A_{Y}+H_{X_{V}}$ is ample by ($\spadesuit$), we see that $-K_{X_{V}}$ is ample. 
By (2), we have 
$$K_{Y}+T+A_{Y}+(p_{V}\circ f)^{*}H_{V}=g_{V}^{*}(K_{X_{V}}+g_{V*}A_{Y}+H_{X_{V}}).$$ 
Moreover, we have $K_{Y}+T=g_{V}^{*}K_{X_{V}}$ by (3) and (4). 
Therefore, $(X_{V},0)$ is lc. 
\end{step1}
We finish the proof. 
\end{proof}

Since $X_{V}$ is lc Fano, there is a $\mathbb{Q}$-divisor $\Delta_{X_{V}}\geq 0$ on $X_{V}$ such that $K_{X_{V}}+\Delta_{X_{V}}\sim_{\mathbb{Q}}0$ and the pair $(X_{V},\Delta_{X_{V}})$ is lc. 
Let $\Delta_{X_{W}}$ be the birational transform of $\Delta_{X_{V}}$ on $X_{W}$. 
Then $K_{X_{W}}+\Delta_{X_{W}}\sim_{\mathbb{Q}}0$ and $(X_{W},\Delta_{X_{W}})$ is lc. 
\end{const}

From now on, we introduce various examples with this construction. 
In the rest of discussions, (i), (ii), (iii), and (iv) denotes the properties (i), (ii), (iii), and (iv) of Proposition \ref{prop--property} respectively. 
Note that although we do not take $V$ explicitly in the examples below, we can take $\mathbb{P}^{n}$ as $V$.

\begin{exam}\label{exam--2}
First we give a flop for a $\mathbb{Q}$-factorial lc pair which is not symmetric. 

Take $W$ as in Construction such that $\rho(W)=1$ and $H^{1}(W,\mathcal{O}_{W})\neq0$, for example, take $W$ as an elliptic curve.
We consider the diagram. 
\begin{equation*}
\xymatrix
{
X_{W}\ar@{-->}[rr]\ar[dr]_{\pi_{W}}&&X_{V} \ar[dl]^{\pi_{V}}\\
&X
}
\end{equation*}
Then $\rho(X_{W}/X)=1$. 
Indeed, by (i), we have $\rho(X_{W})=2$. 
We have $1\leq \rho(X)<\rho(X_{W})$, so $\rho(X)=1$. 
Since $\rho(X_{W}/X)\leq \rho(X_{W})-\rho(X)$, we see that $\rho(X_{W}/X)=1$. 
By (ii), we see that the morphism $\pi_{W}$ is a $\bigl(g_{W*}(p_{V}\circ f)^{*}H_{V}\bigr)$-flipping contraction and $\pi_{V}$ is the flip of $\pi_{W}$. 
Therefore, the diagram is a flop for $K_{X_{W}}+\Delta_{X_{W}}$. 
Moreover, by (i), $X_{W}$ is $\mathbb{Q}$-factorial. 
On the other hand, $X_{V}$ is not $\mathbb{Q}$-factorial. 
Indeed, with notation as in Construction, pick a non-torsion $M\equiv 0$ on $W$ and take the birational transform $M_{X_{V}}$ of $h_{W}^{*}M$ on $X_{V}$. 
If $X_{V}$ is $\mathbb{Q}$-factorial, then $M_{X_{V}}\equiv 0$, so $M_{X_{V}}\sim_{\mathbb{Q}} 0$ by \cite[Theorem 13.1]{fujino-fund} and (iv). 
Then $h_{W}^{*}M\sim_{\mathbb{Q}}0$, hence $M\sim_{\mathbb{Q}}0$ which contradicts our choice of $M$. 
Thus, $X_{V}$ is not $\mathbb{Q}$-factorial. 
In particular, the flop is not symmetric by Lemma \ref{lem--symflop}. 
In this way, the above example shows that flops are not always symmetric. 
So the flops does not satisfies the condition ($*$) in Theorem \ref{thm--main}. 
\end{exam}

\begin{rem}\label{rem--inverseflop} 
We give a further remark on Example \ref{exam--2}. 
Since $\rho(X_{W})=2$ and $\rho(X_{W}/X)=1$, by Lemma \ref{lem--picard}, we have $\rho(X_{V})=2$ and $\rho(X_{V}/X)=1$. 
We see that $K_{X_{W}}$ is ample over $X$ since $0\sim_{\mathbb{Q}}K_{X_{W}}+g_{W*}A_{Y}+g_{W*}(p_{V}\circ f)^{*}H_{V}\sim_{\mathbb{Q},X} K_{X_{W}}+g_{W*}(p_{V}\circ f)^{*}H_{V}$ ((2) and (3)) and $-g_{W*}(p_{V}\circ f)^{*}H_{V}$ is ample over $X$ by (ii). 
Then the diagram
\begin{equation*} \xymatrix { X_{V}\ar@{-->}[rr]\ar[dr]_{\pi_{V}}&&X_{W} \ar[dl]^{\pi_{W}}\\ &X } \end{equation*}
is a $K_{X_{V}}$-flip and a flop for $K_{X_{V}}+\Delta_{X_{V}}$. As we have seen in Example \ref{exam--2}, the flop is not symmetric. In this way, this example shows that flops are not always symmetric even if it is a flip for canonical divisor and the differences between Picard numbers of varieties of the flipping contraction and the flip are one. \end{rem}

\begin{exam}\label{exam--3}
Next, we introduce an example which shows that the assumption on the existence of an open set $U$ in the hypothesis of Theorem \ref{thm--main} and Theorem \ref{thm--pic0} is necessary to obtain the same conclusions as in Theorem \ref{thm--main} and Theorem \ref{thm--pic0}.

In the same way as in Example \ref{exam--2}, we take $W$ as in Construction such that $\rho(W)=1$ and $H^{1}(W,\mathcal{O}_{W})\neq0$. 
We consider the small birational map 
$$X_{V}\dashrightarrow X_{W}$$
and the pairs $(X_{V},\Delta_{X_{V}})$ and $(X_{W},\Delta_{X_{W}})$ which are log minimal models of $(X_{V},\Delta_{X_{V}})$. 

Firstly, we check that the same conclusion as in Theorem \ref{thm--main} does not hold for the lc pairs $(X_{V},\Delta_{X_{V}})$ and $(X_{W},\Delta_{X_{W}})$. 
Note that all small birational morphisms appearing in the conclusion of Theorem \ref{thm--main} are extremal contractions. 
So, if the same conclusion as in Theorem \ref{thm--main} holds for $(X_{V},\Delta_{X_{V}})$ and $(X_{W},\Delta_{X_{W}})$, then the birational transforms of all numerically trivial $\mathbb{Q}$-Cartier divisors by $X_{W}\dashrightarrow X_{V}$ must be $\mathbb{Q}$-Cartier. 
As in the argument in Example \ref{exam--2}, with notation as in Construction, we pick a non-torsion Cartier divisor $M\equiv 0$ on $W$, and consider the birational transform $M_{X_{V}}$ of $h_{W}^{*}M$ on $X_{V}$. 
As in the argument in Example \ref{exam--2}, we see that $M_{X_{V}}$ is not $\mathbb{Q}$-Cartier. 
In this way, the same conclusion as in Theorem \ref{thm--main} does not hold for $(X_{V},\Delta_{X_{V}})$ and $(X_{W},\Delta_{X_{W}})$. 

Secondly, we check that all the statements as in Theorem \ref{thm--pic0} do not hold for the lc pairs $(X_{V},\Delta_{X_{V}})$ and $(X_{W},\Delta_{X_{W}})$. 
By (i), $X_{W}$ is $\mathbb{Q}$-factorial and $\rho(X_{W})=2$. 
The argument in the previous paragraph shows that $X_{V}$ is not $\mathbb{Q}$-factorial. 
From these facts, if $X_{V}$ has a small $\mathbb{Q}$-factorialization $\tilde{X}\to X_{V}$, then $\rho(\tilde{X})>\rho(X_{V})>\rho(X)=1$. 
Therefore we have $\rho(\tilde{X})\geq3$. 
On the other hand, we have $\rho(\tilde{X})\leq \rho(X_{W})=2$ by Lemma \ref{lem--picard}, which is a contradiction. 
From this, we see that $X_{V}$ cannot have a small $\mathbb{Q}$-factorialization. 
In this way, the first assertion of Theorem \ref{thm--pic0} does not hold for $(X_{V},\Delta_{X_{V}})$ and $(X_{W},\Delta_{X_{W}})$. 

The variety $X_{V}$ is lc Fano (see (iv)), so ${\rm dim}H^{1}(X_{V},\mathcal{O}_{X_{V}})=0$ by \cite[Theorem 5.6.4]{fujino-book}. 
On the other hand, we have ${\rm dim}H^{1}(X_{W},\mathcal{O}_{X_{W}})={\rm dim}H^{1}(W,\mathcal{O}_{W})>0$ by (iii). 
Then 
$${\rm dim}H^{1}(X_{V},\mathcal{O}_{X_{V}})\neq {\rm dim}H^{1}(X_{W},\mathcal{O}_{X_{W}})$$
and the second assertion of Theorem \ref{thm--pic0} does not hold for $(X_{V},\Delta_{X_{V}})$ and $(X_{W},\Delta_{X_{W}})$. 

In the third paragraph we have checked that $h_{W}^{*}M\equiv 0$ is Cartier but $M_{X_{V}}$ is not even $\mathbb{Q}$-Cartier. 
So the third assertion of Theorem \ref{thm--pic0} does not hold. 
\end{exam}

\begin{exam}\label{exam--4}
This example shows that to prove the second assertion of Theorem \ref{thm--pic0}, the assumption on existence of an open subset as in Theorem \ref{thm--pic0} is necessary even if both varieties of the lc pairs are $\mathbb{Q}$-factorial.  

Take $W$ as in Construction such that ${\rm Pic}(W)\simeq \mathbb{Z}$ and ${\rm dim}H^{p}(W,\mathcal{O}_{W})>0$ for some $p>0$. 
For example, we take $W$ as a smooth hypersurface of degree $d+2$ in $\mathbb{P}^{d+1}$ for some $d\geq3$ (then ${\rm Pic}(W)\simeq{\rm Pic}(\mathbb{P}^{d+1})\simeq \mathbb{Z}$ by Grothendieck--Lefschetz theorem, and ${\rm dim}H^{d}(W,\mathcal{O}_{W})=1$ by $K_{W}\sim0$ and Serre duality). 
We consider the birational map 
$$X_{V}\dashrightarrow X_{W}.$$
Then the lc pairs $(X_{V},\Delta_{X_{V}})$ and $(X_{W},\Delta_{X_{W}})$ are log minimal models of $(X_{V},\Delta_{X_{V}})$, and $X_{W}$ is $\mathbb{Q}$-factorial by (i). 
In the same way as in Step \ref{step1--const}, we can prove that $X_{V}$ is $\mathbb{Q}$-factorial. 
Indeed, with notation as in Construction, for any Weil divisor $\tilde{D}$ on $X_{V}$ we can find a Cartier divisor $\tilde{D}'$ on $V\times W$ such that $\tilde{D}\sim g_{V*}f^{*}\tilde{D}'$ as Weil divisors. 
Because ${\rm Pic}(W) \simeq \mathbb{Z}$, $V$ is a smooth Fano variety, and $W$ is smooth, we can find $\tilde{\alpha}, \tilde{\beta}\in \mathbb{Z}$ and a Cartier divisor $\tilde{D}_{V}$ on $V$ such that $\tilde{\alpha} \neq 0$ and $\tilde{\alpha} \tilde{D}'\sim p_{V}^{*}\tilde{D}_{V}+\tilde{\beta} p_{W}^{*}H_{W}$ as Cartier divisors (\cite[III, Exercise 12.6 (b)]{hartshorne}).  
Then
$$\tilde{\alpha} \tilde{D}\sim g_{V*}f^{*}(p_{V}^{*}\tilde{D}_{V}+\tilde{\beta} p_{W}^{*}H_{W})=h_{V}^{*}\tilde{D}_{V}+\tilde{\beta} g_{V*}(p_{W}\circ f)^{*}H_{W}$$
as Weil divisors, and the right hand side is $\mathbb{Q}$-Cartier. 
So $X_{V}$ is $\mathbb{Q}$-factorial. 
The variety $X_{V}$ is lc Fano by (iv), then ${\rm dim}H^{p}(X_{V},\mathcal{O}_{X_{V}})=0$ by \cite[Theorem 5.6.4]{fujino-book}. 
On the other hand, we have ${\rm dim}H^{p}(X_{W},\mathcal{O}_{X_{W}})={\rm dim}H^{p}(W,\mathcal{O}_{W})\neq0$ by our choice of $W$. 
\end{exam}

\begin{rem}[cf.~{\cite[Example 3.13.9]{fujino-book}}]\label{rem--lcflip}
As in Example \ref{exam--4}, take $W$ as in Construction such that ${\rm Pic}(W)\simeq \mathbb{Z}$ and ${\rm dim}H^{q}(W,\mathcal{O}_{W})>0$ for some $q>0$. 
As in Remark \ref{rem--inverseflop}, we have $\rho(X_{V}/X)=1$ and the diagram 
\begin{equation*}
\xymatrix
{
X_{V}\ar@{-->}[rr]\ar[dr]_{\pi_{V}}&&X_{W} \ar[dl]^{\pi_{W}}\\
&X
}
\end{equation*}
is a flip for the $\mathbb{Q}$-factorial lc pair $(X_{V},0)$.   
The variety $X_{V}$ is lc Fano by (iv), but $X_{W}$ is not even lc Fano type, i.e., there is no divisor $G\geq0$ on $X_{W}$ such that $(X_{W},G)$ is lc and $-(K_{X_{W}}+G)$ is ample. 
Indeed, if $X_{W}$ is of lc Fano type, then we must have ${\rm dim}H^{q}(X_{W},\mathcal{O}_{X_{W}})=0$ by \cite[Theorem 5.6.4]{fujino-book}, which contradicts our choice of $W$. 
So $X_{W}$ is not of lc Fano type. 
We have $0={\rm dim}H^{q}(X_{V},\mathcal{O}_{X_{V}})$ by \cite[Theorem 5.6.4]{fujino-book}. 

In this way, the above flip shows the following facts:
\begin{itemize}
\item
Flips for $\mathbb{Q}$-factorial lc pairs do not always preserve dimension of cohomology of the structure sheaf, and
\item
Flips for $\mathbb{Q}$-factorial lc pairs do not always keep property of being lc Fano type. 
\end{itemize}
\end{rem}

\begin{exam}\label{exam--6}
Finally, we give two log minimal models of a projective lc pair such that the Cartier indices of log canonical divisors are different. 

We take $W$ as in Construction such that $H^{0}(W,\mathcal{O}_{W}(K_{W}))=0$, for example, we take $W$ as an Enriques surface. 
We take $V$ so that $-K_{V}$ is a base point free Cartier divisor. 
By construction, we have
$$K_{Y}+T+A_{Y}-(p_{V}\circ f)^{*}K_{V}-(p_{W}\circ f)^{*}K_{W}\sim0$$
and both $A_{Y}$ and $-(p_{V}\circ f)^{*}K_{V}$ are base point free Cartier divisors. 
We pick a general member $G_{Y}$ in the linear system of $A_{Y}-(p_{V}\circ f)^{*}K_{V}$ such that $(Y,T+G_{Y})$ is lc. 
Let $G_{X_{W}}$ and $G_{X}$ be the birational transforms of $G_{Y}$ on $X_{W}$ and $X$, respectively. 
Because $K_{W}\sim_{\mathbb{Q}}0$, we obtain $K_{X_{W}}+G_{X_{W}}\sim_{\mathbb{Q}}0$ and 
$K_{X}+G_{X}\sim_{\mathbb{Q}}0$, and both $(X_{W}, G_{X_{W}})$ and $(X,G_{X})$ are lc. 
In particular, $(X_{W}, G_{X_{W}})$ and $(X,G_{X})$ are log minimal models of $(X,G_{X})$. 
We prove that $K_{X_{W}}+G_{X_{W}}$ is Cartier but $K_{X}+G_{X}$ is not Cartier. 
We have
$$K_{X_{W}}+G_{X_{W}}\sim g_{W*}(K_{Y}+T+A_{Y}-(p_{V}\circ f)^{*}K_{V})\sim g_{W*}(p_{W}\circ f)^{*}K_{W}=h_{W}^{*}K_{W}.$$
Therefore $K_{X_{W}}+G_{X_{W}}$ is Cartier. 
On the other hand, by construction, $K_{Y}+T+G_{Y}$ is equal to the pullback of $K_{X}+G_{X}$. 
Recall that the image of $T$ by the morphism $Y\to X$ is a point. 
From these facts, if $K_{X}+G_{X}$ is Cartier then $(K_{Y}+T+G_{Y})|_{T}\sim 0$. 
We also have $K_{Y}+T+G_{Y} \sim (p_{W}\circ f)^{*}K_{W}$, furthermore, $T\simeq V\times W$ which implies that the morphism $(p_{W}\circ f)|_{T} \colon T\to W$ is a contraction. Therefore, we have 
$${\rm dim}H^{0}(W,\mathcal{O}_{W}(K_{W}))={\rm dim}H^{0}(T,\mathcal{O}_{T}((K_{Y}+T+G_{Y})|_{T}))=1,$$
a contradiction. 
In this way, $K_{X}+G_{X}$ is not Cartier. 
From this discussion, we see that Cartier index of $K_{X_{W}}+G_{X_{W}}$ is different from that of $K_{X}+G_{X}$. 
\end{exam}


\end{document}